\documentclass{article}
\usepackage{amsthm, graphicx, amssymb, amsmath, amsfonts, amscd , enumerate, url, epsf}

\newcommand{\C}{\mathbb C}
\newcommand{\A}{\mathbb A}
\newcommand{\PP}{\mathbb P}
\newcommand{\Z}{\mathbb Z}
\newcommand{\Q}{\mathbb Q}

\newcommand{\tha}{^{\text{th}}}

\DeclareMathOperator{\SL}{SL}
\DeclareMathOperator{\PGL}{PGL}

\DeclareMathOperator{\ord}{ord}

\DeclareMathOperator{\Aut}{Aut}

\DeclareMathOperator{\Fix}{Fix}
\DeclareMathOperator{\Per}{Per}
\DeclareMathOperator{\Res}{Res}
\DeclareMathOperator{\ch}{char}
\DeclareMathOperator{\Rat}{Rat}

\newcommand{\ignore}[1]{}

\newtheorem{theorem}{Theorem}[section]
\newtheorem{lemma}[theorem]{Lemma}
\newtheorem{prop}[theorem]{Proposition}

\newtheorem{define}[theorem]{Definition}
\newtheorem{cor}[theorem]{Corollary}

\newtheoremstyle{citing}{}{}{\itshape}{}%
{\bfseries}{.}{ }{\thmnote{#3}}
\theoremstyle{citing}
\newtheorem{cit}{}

\theoremstyle{remark}
\newtheorem*{example}{Example}
\newtheorem*{remark}{Remark}
\newtheorem*{claim}{Claim}

\begin{document}
\title{Moduli spaces for families of rational maps on $\mathbb{P}^1$}
\author{Michelle Manes}
\date{\today}

\maketitle

\begin{abstract}
Let $ \phi:\PP^1\to\PP^1$ be a rational map 
defined over a field~$K$.   We construct the moduli space $M_d(N)$ parameterizing conjugacy classes of degree-$d$ maps with a point of formal period $N$ and present an algebraic proof that $M_2(N)$ is geometrically irreducible for $N > 1$.  Restricting ourselves to maps $\phi$ of arbitrary degree $d\geq 2$ such that $h^{-1}\circ\phi\circ h=\phi$ for some nontrivial $h\in\PGL_2\ \left(\overline K \right) $, we show that the moduli space parameterizing these maps with a point of formal period $N$ is geometrically reducible for infinitely many $N$.
\end{abstract}

\section{Introduction}

 Let $\phi: \mathbb{P}^1 \rightarrow  \mathbb{P}^1$ be a morphism defined over a field~$K$ with algebraic closure $\overline K$.  We denote by $\phi^N$ the $N^{\mathrm{th}}$ iterate of $\phi$ under composition.  A point $P \in \PP^1$ is \emph{periodic} if there exists an integer $N>0$ such that $\phi^N(P)=P$, and $P$ is \emph{preperiodic} if there exist integers $N > M \geq 0$ such that $\phi^N(P) = \phi^M(P)$.  We say $P$ has \emph{period} $N$ if $\phi^N(P) = P$; it has \emph{primitive period} $N$ if $N>0$ is the smallest such integer; and it has \emph{formal period} $N$ if it is a root of the $N^{\mathrm{th}}$ ``dynatomic polynomial'' (defined in Section~\ref{prelim}).  Except in rare cases, points of formal period $N$ coincide with those of primitive period~$N$. 
 
Because we focus on $\PP^1$, every rational map is in fact a morphism; we therefore use the terms  interchangeably.   Further, if we write $\phi$ as a rational map $\phi(z)=F(z)/G(z)$ with $F, G \in K[z]$,  we may take $\deg \phi = \max\{\deg F, \deg G\}$, which corresponds to the usual notion of degree of a morphism of projective curves.

  In Section~\ref{makemd}, we construct the moduli space of rational maps of  degree~$d$ with level-$N$ structure; that is, the space $M_d(N)$ parameterizing rational maps of degree $d$ up to coordinate change together with a point of formal period $N$.   In Section~\ref{qpvir} we prove the following.

\begin{cit} [Theorem~\ref{irred}]
$M_2(N)$ is geometrically irreducible for every $N>1$.
\end{cit}

The proof takes advantage of the fact that we have an explicit description of $M_2$, the moduli space of degree~$2$ rational maps.  Using a normal form for quadratic rational maps, one may iterate to find a polynomial description of a surface mapping surjectively to $M_2(N)$ for each $N$.  We then specialize and apply a result of Morton to prove irreducibility of the covering surfaces, showing that $M_2(N)$ is also irreducible.

There are, of course, natural surjective maps
$$
p_N \colon M_2(N) \to M_2,
$$
simply forgetting the point of period~$N$.  An algebraic curve $C \subseteq M_2$ defines a family of (conjugacy classes of) quadratic rational maps $\phi_C$, and the pullback $p_N^{-1}(C)$ is an algebraic curve on  $M_2(N)$.  One such curve $C$ corresponds to the family of quadratic polynomials, which may be written in the normal form $f_c(z) = z^2+c$.  In his thesis~\cite{bousch}, Bousch proved that the  dynatomic polynomial $\Phi_{N, f_c}^*(z) = \Phi_N^*(z,c) \in \Z[z,c]$  is irreducible for every $N$.  In other words, the curves in $M_2(N)$ lying over that family are all irreducible.
Bertini's theorem, combined with the fact that the surfaces $M_2(N)$ are irreducible, says  that we should expect generic irreducibility of the period-$N$ curves.  

The main result of Section~\ref{red},  then, is potentially surprising.  It says that for the family of rational maps with a nontrivial automorphism,  the dynatomic polynomials  are usually  \emph{reducible} when $N$ is even.  Hence, we should expect the dynamic modular curves, which are defined by $\Phi_N^*=0$, to be reducible for $N$ even.  Of course, the set of maps with nontrivial autmorphisms are a small subset of the space of rational maps, so this is not generic behavior.  This is actually a consequence of a more general result, which we now describe.

Let $\phi\colon \PP^1 \to \PP^1$ be a rational map.  An element $h \in \PGL_2$ acts on $\phi$ via conjugation: $\phi^h = h^{-1}\circ \phi \circ h$.  If $\phi^h = \phi$ we say that $h$ is an automorphism of $\phi$.  When $\phi$ has a $\PGL_2$-automorphism of prime order~$p$, then the group of automorphisms $\Aut(\phi)$ contains a subgroup isomorphic to  $\mathfrak C_p$, the cyclic group of order $p$.  
We define $M_d(N, \mathfrak C_p)$ to be the moduli space of (equivalence classes of) rational maps of degree $d$ having an automorphism of prime order~$p$, together with a  point of primitive period~$N$.

\begin{cit}[Corollary~\ref{mainthmcor}]
$M_d(pN, \mathfrak C_p)$ is reducible for 
\begin{enumerate}[\textup(a\textup)]
\item
all but finitely many integers $N$ if $K$ has characteristic~$0$, and
\item
all but finitely many prime integers $N$ for arbitrary fields $K$.
\end{enumerate}
\textup(In particular, the curves $M_2(2N, \mathfrak C_2)$ are reducible for infinitely many~$N$.\textup)
\end{cit}

This result is dramatically different from reducibility results that have appeared in the literature previously.  If $\phi(z) = z^d$ is a pure power function or if $\phi(z)$ is a Chebyshev polynomial, 
the dynatomic polynomial $\Phi_{N,\phi}^*$ is known to be reducible for infinitely many $N$.  In Section~\ref{power}, we add reciprocals of pure power functions, $ \phi(z) = z^{-d}$, to this list.  In all of these cases, however, the reducibility result holds for only one map (up to conjugacy) of each degree $d$.  In Corollary~\ref{mainthmcor}, we have a natural \emph{family} of maps in each degree such that the dynatomic polynomials are reducible infinitely often for every map in the family.  
  
 The proof of Corollary~\ref{mainthmcor} is constructive, providing a particular proper closed subvariety of  $M_d(pN, \mathfrak C_p)$ of maximal dimension.  The difficulty of the proof lies in first defining the appropriate object, and then proving that the object is in fact a proper subvariety, which is not at all obvious from the definition.

The construction  gives a geometric explanation for reducibility when a rational map $\phi$ has a nontrivial automorphism:  If $P$ is a point of formal period $N$, then $h(P)$ must be as well for $h \in \Aut(\phi)$.  There are two possible cases if the order of $h$ divides $N$: either $P$ and $h(P)$ are on the same orbit, or the action of $h$ interchanges the separate orbits of $P$ and $h(P)$.  We show that the moduli space has at least two components, corresponding to these two possibilities.
 
The dynatomic polynomials allow us to relate $K$-rational periodic points for a morphism $\phi$ to $K$-rational points on an algebraic curve.  Understanding the geometric properties of these curves (reducibility, genera, etc.) has allowed authors to tackle questions related to the  uniform boundedness conjecture in arithmetic dynamics (see~\cite{quadmaps2} and \cite{5cyc}) as well as more general number theoretic questions related to cyclic extensions of number fields (see~\cite{algcurve}).  To date, most of this work has been done with polynomial maps.  It is our hope that this paper will spur work on more general rational maps, and in fact some of this work has already appeared (see~\cite{manesQcyc}).

\begin{remark}
This paper forms a portion of the author's Ph.D.~thesis~\cite{manesthesis}.
\end{remark}

\section{Preliminaries}\label{prelim}

Let $X$ be a variety and 
$
\phi \colon X \to X
$
be a morphism defined over some field $K$.  Let 
$$
\Per_N(\phi) = \{\alpha \in X\left(\overline K\right): \phi^N(\alpha) = \alpha\}
$$
 be the set of points of period $N$ for $\phi$.
We define
\begin{align*}
\Delta=\Delta(X) &= \{(x,x) : x \in X\} \subset X \times X,
&& \text{ the diagonal, }\\
\Gamma(\phi^N) & =  \{(x,\phi^N(x)) : x \in X\} \subset X \times X,
&& \text{ the graph of the morphism $\phi^N$.}
\end{align*}
We can then assign a multiplicity to each  $P \in \Per_N(\phi)$ by taking the intersection multiplicity $a_{P}(N)$ of the diagonal with $\Gamma(\phi^N)$ in $X \times X$.  Following Morton and Silverman in~\cite{dynunit}, we define the cycle of $N$-periodic points
\begin{equation}\label{zn}
Z_N(\phi) = \sum_{P \in X\left(\overline K\right)} a_P(N) P
\overset{\text{def}}{=}
\Delta \cdot \Gamma(\phi^N),
\end{equation}
and the cycle of primitive $N$-periodic points
\begin{equation}\label{zn*}
Z_N^*(\phi) = \sum_{P \in X\left(\overline K\right)} a_P^*(N) P
\overset{\text{def}}{=}
\sum_{k\mid N}\mu\left(\frac N k\right)Z_k(\phi),
\end{equation}
where $\mu$ is the Moebius function.  In~\cite{dynunit}, the authors show that if $X$ is a curve, then $Z_N^*$ is an effective $0$-cycle, and they give a precise description of the points $P \in X$ with $a_P^*(N) >0$.   In his thesis, Hutz~\cite{hutz1} has extended these results to $X$ an irreducible, nonsingular projective variety of arbitrary dimension.

If $\alpha\in \Per_N(\phi)$ then $\phi^N$  induces a map from the cotangent space of $X$ to itself,
$$
\left(\phi^N\right)^* \colon \Omega_{\alpha}(X) \longrightarrow \Omega_{\alpha}(X).
$$
If $X$ is a smooth curve, then the cotangent space has dimension one.  So $\left(\phi^N\right)^*$ must be multiplication by a scalar, which we call the \emph{multiplier}\label{multdef} of the cycle associated to $\alpha$.
When $X = \PP^1$, we may write $\phi(z) \in K(z)$ as a rational map.  Then the scalar is exactly $\left(\phi^N\right)'(\alpha)$ as long as   the point at infinity is not in the orbit of $\alpha$ (though there is a natural extension to this case, see~\cite[exercise 1.13]{ads}).  In particular, for $\phi \colon \PP^1 \to \PP^1,$
and for $\alpha\in \PP^1$ a fixed point of $\phi$, the multiplier of the fixed point is $\phi'(\alpha)$.

Note that if $\deg(\phi) = d$, then $\phi$ has $d+1$ fixed points, counted with proper multiplicity.   When $X = \PP^1$ and the  fixed points of $\phi$ are all distinct, there is an identity on the multipliers of these fixed points.
 Let $\lambda_1, \cdots, \lambda_{d+1}$ be the multipliers.  Then by~\cite[Theorem 1.14]{ads}, we have
\begin{equation}\label{fpident}
\sum_{i=1}^{d+1}
1/(1-\lambda_i)  = 1.
\end{equation}
The requirement that the fixed points are distinct means that none of the $\lambda_i$ are~$1$.  If the fixed points are not all distinct, a more complicated identity still holds (see~\cite[exercise~$1.17$]{ads}).

If $X = \PP^1$, then we may write the morphism $\phi \colon\PP^1 \to \PP^1$ using homogeneous coordinates
\begin{align}
\phi(x,y) &=\left[F_1(x,y): G_1(x,y)\right] \nonumber\\
&= \left[a_d x^d + \cdots + a_{1}x y^{d-1}+a_0 y^d: b_d x^d  +\cdots+ b_{1}x y^{d-1}+ b_0y^d\right],\label{phiform}
\end{align}
for some polynomials $F_1, G_1 \in K[x,y]$ with no common factors over $\overline K$, and $\deg \phi = \deg F_1 = \deg G_1$.   Of course, for any $u \in \overline K^*$,   $[uF_1 : uG_1]$ defines the same rational map $\phi$.  We therefore identify each $\phi$ with a unique point in $\PP^{2d+1}$ via
\[
\phi \longmapsto
\left[a_d:\cdots:a_{1}:a_0:b_d:\cdots:b_{1} : b_0\right].
\]
The requirement that $F_1$ and $G_1$ share no common factors means, however, that not every point in $\PP^{2d+1}$ corresponds to a map of degree~$d$.

The \emph{resultant} of two polynomials $F$ and $G$ is a polynomial in the coefficients of  $F$ and $G$ with the property that $\Res(F,G) = 0$ if and only if $F$ and $G$ have a common zero in $\PP^1\left(\overline K\right)$.  (See \cite[Proposition 2.13]{ads} for details.)  So the space of rational maps of degree~$d$ corresponds to an affine variety:
\[
\Rat_d =
\PP^{2d+1} \smallsetminus \{\Res(F_1,G_1)=0\}.
\]

 We say that two rational maps $\phi$ and $\psi$ are \emph{linearly conjugate} if there is some $h\in \PGL_2\left(\overline K\right)$ such that $\phi^h = \psi$. 
Because linearly conjugate maps have the same dynamical behavior, it is natural to consider the quotient space 
\[
M_d = \Rat_d /\PGL_2.
\]
Generalizing 
work by Milnor~\cite{milnrat}, Silverman~\cite{mdgit} proved that $M_d$ exists as an affine integral scheme over $\Z$ and that 
$M_2$ is isomorphic to $\A^2_{\Z}$.   In fact, if we let $\lambda_1, \lambda_2, \lambda_3$ be the multipliers of the three fixed points of $\phi$ (counted with multiplicity), then the first two symmetric functions of these multipliers form natural coordinates for $M_2$:
\begin{equation}\label{m2coords}
M_2  =\{ (\sigma_1, \sigma_2)\} \text{ where } 
\sigma_1 = \lambda_1 + \lambda_2 + \lambda_3, \text{ and }
\sigma_2 =  \lambda_1 \lambda_2 + \lambda_1 \lambda_3 +\lambda_2 \lambda_3.
\end{equation}

Writing $\phi(x,y) = \left[F_1(x,y): G_1(x,y)\right]$, we are able to describe iteration:   
\[
\phi^N(x,y) = \left[F_N(x,y) : G_N(x,y)\right]
\]
 represents $\phi$ composed with itself $N$ times, where the polynomials $F_N$ and $G_N$ are given by the double recursion
\[
  F_N(x,y) = F_{N-1}\left(F_1(x,y),G_1(x,y)\right) \text{ and } G_N(x,y) = G_{N-1}(F_1(x,y),G_1(x,y)).
\]
We now define a homogeneous polynomial $\Phi_{N,\phi}(x,y)$ whose roots are precisely points of period $N$ for $\phi$:
\[
\Phi_{N,\phi}(x,y) =y F_N(x,y)  - x G_N(x,y).
\]
  If $P=[x:y] \in\PP^1$ is a root of this polynomial, then by construction $\phi^N(P) = P$.

\begin{define}
The  $N\tha$ \emph{dynatomic polynomial} for $\phi$ is given by
\[
\Phi_{N,\phi}^*(x,y) =  \prod_{k \mid N} \left(\Phi_{k,\phi}(x,y)\right)^{\mu\left(N/k\right)} =
\prod_{k \mid N} \left(y F_k(x,y)  - x G_k(x,y)\right)^{\mu(N/k)}.
\]
\end{define}

  The cycle $Z_N^*(\phi)$ in equation~\eqref{zn*} is a formal sum of the roots of $\Phi_{N,\phi}^*(x,y)$ counted  with multiplicity, and the result that $Z_N^*(\phi)$ is an effective $0$-cycle means that $\Phi_{N,\phi}^*(x,y)$ is a polynomial.  
   (To ease notation, we will write simply $\Phi_N$ and $\Phi_{N}^*$ unless the distinction is needed.)

 We would like to say that the roots of $\Phi_{N}^*$ are the points in $\PP^1$ with primitive period $N$ for $\phi$.  All such points are, indeed, roots of $\Phi_{N}^*$, but it is possible that other points with smaller primitive period arise as roots as well.  We call the roots of $\Phi_{N}^*$ the points of \emph{formal period} $N$ for $\phi$.
  
  \begin{example}
  Let 
  \begin{align*}
  \phi(x,y) &= \left[-x^2+y^2: xy\right], \text{ so}\\
   \phi^2(x,y) &= [-\left(x^2+xy -y^2\right) \left(x^2- xy -y^2\right): -xy(x+y)(x-y)].
   \end{align*}
    Then we have:
  \begin{align*}
  \Phi_1^*(x,y) & = 
  -y \left(2 x^2-y^2\right)\\
  \Phi_2^*(x,y) & = 
  -y^2
  \end{align*}
  So we see that the point at infinity $\alpha = [1:0]$ is a fixed point, but it also appears as a double-root of the second dynatomic polynomial.
  
  \end{example}

 \bigskip
 
If  $K$ is a field with characteristic different from~$2$, and $f(z) \in K[z]$ is a quadratic polynomial, then $f$ is linearly conjugate to $f_c(z) = z^2+c$ for some $c\in K$.  In this case, one may consider the dehomogenized  $N\tha$ dynatomic polynomial, 
$$
\Phi_{N,f_c}^*(z) =
\prod_{k \mid N} \left( f_c^k(z)  -z \right)^{\mu(N/k)}.$$

\begin{theorem}[Bousch]\label{bousch}
Let $f_c(z) = z^2 +c $.  Then 
$\Phi_{N,f_c}^*(z)  \in \Z[z,c]$ is irreducible over $\C[z,c]$ for every $N$.
\end{theorem}

This result,  proved by Bousch in his thesis~\cite{bousch}, was later generalized by Morton in~\cite{algcurve} to polynomials $f(z,c) \in \Z[z,c]$ of arbitrary degree~$d\geq 2$, satisfying certain conditions, the most important being that the primitive $N$-bifurcation points --- that is, the values of $c \in \C$ for which two $N$-orbits coincide --- are distinct.  (This is well-known to be the case for the quadratic family $f_c$ by early results of Douady and Hubbard on the structure of the Mandelbrot set.)

Viewing $ \Phi_{N,f_c}^*(z) =  \Phi_{N,f_c}^*(z,c)$ as a polynomial in two variables, we see that for every $N$,  
\[
Y_1(N) \colon \Phi_{N,f_c}^*(z,c)=0
\]
 defines an affine algebraic curve.  These are modular curves in the sense that they parameterize isomorphism classes of pairs $(f,\alpha)$, where $f(z)$ is a quadratic polynomial and $\alpha$ is a point of formal period $N$ for $\phi$.  Theorem~\ref{bousch} says that these modular curves are geometrically irreducible.
A natural question to ask is how general this result may be.  If we consider other families of degree-$2$ rational maps on $\PP^1$, should we expect these modular curves to be irreducible? In Section~\ref{red} we tackle this question.

\section{Construction of Moduli Spaces}\label{makemd}

We begin by fixing some notation.

\begin{tabular}{l l}
$K$ & a field.\\
$\overline K$ & a (fixed) algebraic closure of $K$.\\
 $\phi(z)\in K(z)$ &  a rational map of degree $d$.\\
 $\phi(x,y) = \left[F_1(x,y):G_1(x,y)\right]$ & homogenization of $\phi(z)$.\\
 $F_N(x,y) = F_{N-1}(F_1(x,y),G_1(x,y))$. &\\
  $G_N(x,y) = G_{N-1}(F_1(x,y),G_1(x,y))$. &\\
$\phi^N(x,y) = \left[F_N(x,y):G_N(x,y)\right] $&  $\phi$ composed with itself $N$ times, \\
&with $F_N$ and $G_N$ as above.\\
$\Phi_{N,\phi}^*(x,y)$ & the $N\tha$ dynatomic polynomial\\
& for $\phi$ as defined in Section~\ref{prelim}.\\
$\PP^1$ & the projective line $\PP^1(\overline K)$.\\
$\PGL_2$ & the projective linear group over $\overline K$.\\
$\Rat_d$ & the space of degree~$d$ rational \\
& maps as described in Section~\ref{prelim}.\\
$M_d$ & the moduli space $\Rat_d/\PGL_2$.\\
$\Aut(\phi) = \{f \in \PGL_2 : \phi^f = \phi \}$ & the automorphism group of $\phi$.\\
\end{tabular}

\begin{define}
We define a level structure on $M_d$ as follows.  
\begin{align*}
\Rat_d(N) &= 
\{\left(\alpha, \phi \right)\in \PP^1 \times \Rat_d \colon 
 \alpha \text{ is a point of formal period $N$ for $\phi$}\}\\
 M_d(N) & = \Rat_d(N)/\PGL_2.
\end{align*}
 \end{define}

 \begin{prop}
For every $N\geq 1$, $ \Rat_d(N)$ exists as a scheme over~$\Z$, and $M_d(N)$ exists as a geometric quotient scheme over~$\Z$.
\end{prop}
 
 \begin{proof}
For convenience, and following~\cite{mdgit}, we use $\SL_2$ rather than $\PGL_2$ for this proof.  Because over any algebraically closed field $\overline K$, the map $\SL_2(\overline K) \to \PGL_2(\overline K)$ is surjective with finite kernel, the results will still be applicable to the moduli spaces as defined above.

In~\cite{mdgit}, $\Rat_d(N)$ is called $\Per_N^*$, and is shown to be a reduced closed subscheme of $\PP^1_{\Rat_d} = \PP^1 \times \Rat_d$.  
A point of $\Rat_d$ is given by $[F_1, G_1] = [a_d:a_{d-1}:\cdots: a_0:b_d:b_{d-1}:\cdots:b_0]$ corresponding to a rational map 
\begin{align*}
\phi \colon \PP^1 &\to \PP^1 \\ 
[x:y] &\mapsto [a_d x^d+ a_{d-1} x^{d-1}y + \cdots + a_0 y^d :
b_d x^d+ b_{d-1} x^{d-1} y + \cdots  + b_0 y^d],
\end{align*}
with $\Res(F_1, G_1) \neq 0$.
A point of $\Rat_d(N)$, is given by $[(X : Y), (F_1, G_1)]$ such that $[F_1, G_1] \in \Rat_d$ and for $\phi$ given by $[F_1: G_1]$ as above,  $\Phi^*_{N,\phi}(X,Y) = 0$.

The proof for $M_d(N)$ follows the techniques of Proposition~4.2 in~\cite{git}.
We define the action of $\SL_2$ on points in $\PP^1_{\Rat_d}$ as follows.
Let 
\[f =
\left( \begin{matrix}
\alpha & \beta \\
\gamma & \delta\\
\end{matrix} \right) \in \SL_2.
\]
The action on points of $\Rat_d(N)$ is
\begin{align*}
[(X, Y) , &(F_1(x,y), G_1(x,y))] \mapsto [( \delta X - \beta Y , -\gamma X+\alpha Y),\\
&  
 \left (\delta F_1(\alpha x + \beta y, \gamma x + \delta y) - \beta G_1(\alpha x + \beta y, \gamma x + \delta y) ,\right.\\
&\left. \qquad -\gamma F_1(\alpha x + \beta y, \gamma x + \delta y)+\alpha G_1(\alpha x + \beta y, \gamma x + \delta y)\right)].
\end{align*}

If $[X:Y]$ is a point of formal period $N$ for $\phi$, then 
$f^{-1}[X:Y] = [\delta X - \beta Y : -\gamma X+\alpha Y]$ is a point of formal period $N$ for $\phi^f = f^{-1} \phi f$, so $\Rat_d(N)$ is an $\SL_2$ invariant set given the action defined above.  

Further, the resultant form $\Res(F_1, G_1)$ is an $\SL_2$-invariant homogeneous polynomial which does not vanish on points of $\Rat_d(N)$.  The stabilizer in $\SL_2$ of any rational map $\phi$ is finite~\cite[Proposition 4.65]{ads}.  
 Therefore, points of $\Rat_d(N)$ are stable under the $\SL_2$ action.  

Since $\Rat_d(N)$ is an $\SL_2$-invariant and $\SL_2$-stable subset of $\PP^1 \times \PP^{2d+1}$, the geometric quotient 
$ \Rat_d(N) / \SL_2$ exists as a scheme over~$\Z$.
\end{proof}

\begin{remark}
Because for every $d \geq 2$, we have rational maps $\phi \in \Rat_d$ with nontrivial $\PGL_2$ automorphisms, we know that $M_d$ cannot be a fine moduli space.  However, from~\cite[Proposition 4.65]{ads}, for every $d\geq 2$, there is an $r_d$ such that if $\phi \in \Rat_d$, then $\left| \Aut(\phi) \right| \leq r_d$.  So for $N$ large enough, we expect that $M_d(N)$ should be a fine moduli space, since no automorphism of $\phi \in \Rat_d$ can fix  all $N$ marked points.
\end{remark}

Let $\mathfrak C_n$ be the subgroup of $\PGL_2$ generated by $z \mapsto \zeta_n z$ for $\zeta_n$ a primitive $n^{\text{th}}$ root of unity.  

\begin{define}
\begin{align*}
\Rat_d(\mathfrak C_n) &= 
\{\phi \in \Rat_d \colon  \mathfrak C_n  \text{ is a subgroup of } \Aut(\phi)\},\\
M_d(\mathfrak C_n) &= \text{the image of } \Rat_d(\mathfrak C_n) \text{ under the projection } \Rat_d \to M_d, \\
\Rat_d(N, \mathfrak C_n) &= 
\{\phi \in \Rat_d(N) \colon  \mathfrak C_n  \text{ is a subgroup of } \Aut(\phi)\},\\
M_d(N, \mathfrak C_n) &= \text{the image of } \Rat_d(N, \mathfrak C_n) \text{ under the projection } \\
&\quad\Rat_d(N) \to M_d(N). \\
\end{align*}
 \end{define}
 
 \begin{prop}
 If $n$ is relatively prime to the characteristic of $K$, then the following hold for all $d \geq 2$:
 \begin{enumerate}[\textup(a\textup)]
 \item
 $\Rat_d(\mathfrak C_n) $ exists as a closed subscheme of $\Rat_d$.
 \item
  $M_d(\mathfrak C_n) $ exists as a closed subscheme of $M_d$.
 \item\label{rdNcn}
 $\Rat_d(N, \mathfrak C_n) $ exists as a closed subscheme of $\Rat_d(N)$.
 \item\label{mdNcn}
  $M_d(N, \mathfrak C_n) $ exists as a closed subscheme of $M_d(N)$.
  \end{enumerate}
\end{prop}

\begin{proof}
A rational map $\phi$ has the automorphism $z \mapsto \zeta_n z$ if $\zeta_n^{-1} \phi(\zeta_n z) = \phi(z)$.  Writing $\phi$ as in~\eqref{phiform}, a straightforward calculation shows that this yields polynomial conditions on the coefficients of $\phi$ and moreover that the polynomials are all defined over $\Q$. 

Since $M_d$ is a geometric quotient, one can show that the projection $\Rat_d \to M_d$ is closed, meaning that the image of $\Rat_d(\mathfrak C_p)$ is closed.  Hence $M_d(\mathfrak C_p)$is a closed subscheme of $M_d$.

Identical arguments prove parts~\eqref{rdNcn} and~\eqref{mdNcn}.
\end{proof}

\begin{prop}
The scheme $M_d(\mathfrak C_n)$ parameterizes conjugacy classes of rational maps of degree~$d$ having an automorphism of primitive order~$n$.  Similarly, $M_d(N, \mathfrak C_n)$ parameterizes conjugacy classes of rational maps of degree~$d$ having an automorphism of primitive order~$n$, together with a point of formal period~$N$.
\end{prop}  

\begin{proof}
A conjugacy class of maps $[\phi]\in M_d$ is in $M_d(\mathfrak C_n)$ exactly when there is some $\psi \in [\phi]$ such that $\mathfrak C_n$ is a subgroup of $\Aut(\psi)$.  Such maps clearly have automorphisms of primitive order $n$.

Now let $[\phi] \in M_d$ such that some $\psi \in [\phi]$ has an automorphism $f \in \PGL_2$ of primitive order $n$.  Then $\langle f \rangle$ is a subgroup of $\Aut(\psi)$ isomorphic to $\mathfrak C_n$.  We will show that in fact $[\phi] \in M_d(\mathfrak C_n)$.

Since $n$ is relatively prime to the characteristic of $K$, we see that $\langle f \rangle$ is linearly conjugate to $\langle \zeta_n z \rangle$ for an appropriate choice of $n$, simply by moving the fixed points of $f$ to~$0$ and the point at infinity.  Thus, for some $h \in \PGL_2$, $\Aut(\psi^h)$ contains $\mathfrak C_n$ as a subgroup.  Since $\psi^h \in [\phi]$, we are done.

The proof for $M_d(N, \mathfrak C_n)$ is the same.
\end{proof}

\section{Irreducibility for $M_2(N)$}\label{qpvir}
We continue with the notation of Section~\ref{makemd}.  In this section only, we assume $K$ is a number field and not just an arbitrary field.
 In addition, we define

\begin{tabular}{l l}
$\nu_d(N)= \displaystyle\sum_{k\mid N} \mu\left(N/k\right) d^k$ & the degree of $\Phi_N^*(x,y) $ for $N>1$.\\
\end{tabular}

In proving their irreducibility results, both Bousch~\cite{bousch} and Morton~\cite{algcurve} use the fact that any quadratic polynomial is conjugate to a unique map of the form $f_c(z) = z^2+c$.  That is, we have a convenient normal form, and hence can prove results about conjugacy classes by working with the normal form alone.  
The situation for quadratic rational maps is slightly more complicated, so we begin this section by discussing the normal form for these maps which will be useful in the sequel. 

\begin{lemma}\label{normform}
Let $\phi\colon\PP^1 \to\PP^1$ be a rational map of degree~$2$ defined over a number field $K$.  If there is an element $h \in \PGL_2$ such that $h^{-1}\phi h = \phi$, then $\phi$ is conjugate to exactly one of the following maps:
\begin{itemize}
\item
$\displaystyle \phi(z) = z + \frac 1 z$, or
\item
$\displaystyle\phi_{a,a}(z) = \frac{z^2+az}{az+1}$, where $a \neq \pm1$.
\end{itemize}
\end{lemma}

Note that $\phi_{a,a}(z)$ may not be defined over $K$, but is defined over at most a cubic extension of $K$.

\begin{proof}
The maps described all have obvious automorphisms $h \in \PGL_2$: In the first case, the automorphism is $z \mapsto -z$; in the second case, it is $z \mapsto 1/z$.  It remains to show that any map with an automorphism is conjugate to exactly one of these maps.

Let $\phi\colon\PP^1 \to \PP^1$ be a rational map of degree 2, defined over $\C$.  Using the fact that such a map is completely determined by the set of multipliers of the fixed points, Milnor has shown in~\cite{milnrat} that $\phi$ must be conjugate to one of the following 
\begin{align*}
\phi(z) &= z + \frac 1 z \quad \text{ (if $\phi$ has a single fixed point), or}\\
\phi_{a,b}(z) &= \frac{z^2+az}{bz+1},
\end{align*}
where $a$ and $b$ are the multipliers of the fixed points $0$ and $\infty$ respectively, and $ab\neq 1$.  

If $\phi$ has a nontrivial automorphism $h$ as described, then $h$ must permute the fixed points of $\phi$.  Taking derivatives on both sides of $\phi h = h\phi$, we see that the multipliers of the interchanged fixed points must be equal. \label{equalmultsfp}

If $\phi$ has a single fixed point, it is a simple matter to check that the multiplier at that fixed point must be $1$ and that $\phi$ is then conjugate to the first map given in the the Lemma.

If $\phi$ has two distinct fixed points, then it has one fixed point of multiplicity two.  This fixed point has multiplier~$1$.
  So the set of multipliers is $\{a, 1, 1\}$ where $a\neq 1$.  By Milnor's argument, then, $\phi$ is conjugate to the map $\phi_{a,1}= \frac{z^2+az}{z+1}$, which has a double fixed point at infinity and another fixed point at~$0$.  Any automorphism of such a map must fix both of these points, so it is of the form $h(z) = k z$, and a computation shows that we must have $k=1$.  In other words, these maps have no nontrivial automorphisms.

Finally, if $\phi$ has three distinct fixed points, we know from the remarks above that at least two of them must have equal multipliers.  The identity given in equation~\eqref{fpident} on page~\pageref{fpident} shows that it is impossible for the two equal multipliers to be~$-1$.  If a fixed point has multiplier~$1$, then it is a double-root of the equation $\Phi_1(x,y)=0$.  In other words,  it is a fixed point of multiplicity two, which cannot be the case here.  So the set of multipliers is $\{a,a,b\}$ with $a \neq \pm1$.  By Milnor's argument, then, $\phi$ is conjugate to a unique map of the form $\phi_{a,a}$.
\end{proof}

The map $\phi(z) = z + \frac 1 z$ has a single fixed point at infinity, with multiplier~$1$, so using the coordinates given by equation~\eqref{m2coords}, this map corresponds to the point $(3,3)  \in M_2$.   Therefore, we have the following quasi-finite map:
\begin{align*}
\A^2 \smallsetminus \{ab = 1\} &\longrightarrow M_2\smallsetminus \{(3,3)\}\\
(a,b) &\longmapsto \left[\phi_{a,b} \right].
\end{align*}
Here $ \left[\phi_{a,b}\right]$ denotes the conjugacy class of $\phi_{a,b}(z)$ in the moduli space $M_2$.  (Note that the map is generically $6$-to-$1$, but is not a finite map.)
This map is ramified only over the symmetry locus --- the family of maps in $M_2$ with a nontrivial $\PGL_2$ automorphism --- and over the locus of maps with a fixed point of multiplicity at least two.

\begin{lemma}\label{leadx}
Let 
$\phi_{a,b}(x,y) = \left[x^2 + axy: bxy + y^2\right]$. 
 For all $N>1$, the  coefficient of $x^{\nu_2(N)}$ in $\Phi_{N,\phi_{a,b}}^*$ \textup(that is, the lead coefficient in $x$\textup) is $C_N(b)$, the $N^{\text{th}}$ cyclotomic polynomial in $b$.  By symmetry, the coefficient of $y^{\nu_2(N)}$ is $ C_N(a)$.
\end{lemma}

\begin{proof}
Let
$F_1(x,y) = x^2 + axy$, $G_1(x,y) = bxy + y^2$,  and define the iterates
\begin{eqnarray}
F_N(x,y) &=& F_{N-1}^2 + a F_{N-1}G_{N-1},\label{fniterates}\\
G_N(x,y) &=& b F_{N-1} G_{N-1}  +  G_{N-1}^2\label{gniterates}.
\end{eqnarray}

\begin{claim}
As polynomials in $x$, the iterate $F_N$ is monic of degree $2^N$, and $G_N$ has degree $2^N - 1$ with lead coefficient $b^N y$.
\end{claim}

A simple induction argument shows that for $N \geq 1$, we have $\deg_x(F_N) = \deg_x(G_N)+1$.   
The claim about the degrees then follows immediately from this and  equations~\eqref{fniterates} and~\eqref{gniterates}.  The fact that $F_N$ is monic in $x$ also follows  from the recursive definition.  It remains to compute the coefficient of $x^{2^N-1}$ in $G_N(x,y)$.   The claim holds for $N = 1$ by  definition. 
Further, the lead $x$ term in $G_N(x,y)$ comes from the product  $b F_{N-1}G_{N-1}$.  We again proceed by induction.  
\begin{align*}
G_N(x,y) &= b\left(x^{2^{(N-1)}} + \text{lower order terms in $x$}\right)\left(b^{N-1}yx^{2^{(N-1)}-1} + \text{l.o.t. in $x$}\right)\\
&   \qquad  + \left(\text{l.o.t. in $x$} \right) \\
&= b^N y x^{2^N-1} + \text{lower order terms in $x$.} 
\end{align*}

Let  $c_k$ be the coefficient of $x^{2^k}$ in $\Phi_{k,\phi_{a,b}}(x,y) = y F_k(x,y)  - x G_k(x,y)$ and $c_k^*$ be the coefficient of  $x ^{\nu_2(k)}$ in $\Phi_{k,\phi_{a,b}}^*$.  By the above results, $c_k = y(1- b^k)$.   Since 
$$
\Phi_{N,\phi_{a,b}}^*(x,y) = \prod_{k \mid N} \left(y F_k(x,y)  - x G_k(x,y)\right)^{\mu(N/k)},
$$
 and since $\Phi_{N,\phi_{a,b}}^*(x,y)$ is a polynomial,
 \begin{align*}
c_N^* &= \prod_{k \mid N} c_k^{\mu(N/k)}  =  \prod_{k \mid N} \left(y(1-b^k)\right)^{\mu(N/k)}\\
& =  \prod_{k \mid N} y^{\mu(N/k)} \prod_{k \mid N}(1-b^k)^{\mu(N/k)}= \prod_{k\mid N}(b^k-1)^{\mu(N/k)},
\end{align*}
 where the last equality follows because $N>1$.   This is $C_N(b)$, as desired.
\end{proof}

Mobius inversion gives $\Phi_{N,\phi_{a,b}} = \prod_{M\mid N}\Phi_{M,\phi_{a,b}}^*(x,y) $.  In other words, $\Phi_{N,\phi_{a,b}}$ factors over $\Q(a,b)$.
Gauss's lemma shows that it factors over $\Z[a,b]$.   By Lemma~\ref{leadx} above, we see that the polynomials $\Phi_{M,\phi_{a,b}}^*(x,y)$ in the product each have content~$1$, meaning that $\Phi^*_{N,\phi_{a,b}}\in \Z[a,b,x,y]$.
So for every $N$, 
$$
\left\{\Phi_{N,\phi_{a,b}}^*\left(x,y,a,b\right)=0\right\} \subseteq \PP^1 \times \A^2\smallsetminus\{ab=1\}
$$
  defines a quasi-projective variety over $\overline K$.  A point on the variety determines a rational map $\phi_{a,b}$ and a point of formal period $N$ for that map.

\begin{prop}\label{irred2}
Let $\phi(x,y)   = \left[x^2+axy : bxy + y^2\right]$.  The polynomial
\[
\Phi_{N,\phi}^*(x,y,a,b)\in \overline\Q[x,y,a,b]
\]
 is irreducible for all $N>1$.
\end{prop}

The proof of Proposition~\ref{irred2} requires a result of Morton.
\begin{lemma}[Corollary~4 in~\cite{algcurve}] \label{morcor}
For $N>1$, the polynomial $\Phi_{N,h}^*(z,a)\in \overline \Q[z,a]$ corresponding to $h(z,a) = z^2+az$ is irreducible.
\end{lemma}

\begin{proof}[Proof of Proposition~\ref{irred2}]
Suppose  for contradiction that 
$$
\Phi_N^*(x,y,a,b) = A(x,y,a,b)B(x,y,a,b) \quad \text{with} \quad \deg_x(A), \deg_x(B) \geq 1.
$$
  By Lemma~\ref{leadx}, the  coefficient of $x^{\nu_2(N)}$ in $\Phi_{N}^*(x,y,a,b)$ is $C_N(b)$.  So the coefficients of the lead $x$ terms in $A$ and $B$ are $c_{N,A}(b)$ and  $c_{N,B}(b)$, where  
  $$c_{N,A}(b) c_{N,B}(b)= C_N(b).$$

Now we specialize to $b=0$.  For $N>1$,  $C_N(0) =1$, so  $c_{N,A}(0)c_{N,B}(0)\neq 0$.
Hence, 
\begin{equation}\label{conteq}
\Phi_N^*(x,y,a,0) = A(x,y,a,0) B(x,y,a,0),
\end{equation}
where neither $A$ nor $B$ is trivial.  

But the specialization to $b=0$ yields the map $\phi_{a,0}(x,y) = [x^2+axy:y^2]$, which when dehomogenized is exactly the polynomial $h(z,a)$ from Corollary~4 in~\cite{algcurve}, stated above.  Hence $\Phi_N^*(x,y,a,0)$ is irreducible,  contradicting equation~\eqref{conteq}.  
\end{proof}

The condition $n > 1$ is necessary, as 
$$\Phi_1^*(x,y,a,b) = x y (x(1 - b ) + y(1 -a ) ).
$$
  This is expected, because part of constructing Milnor's normal form involves moving two of the fixed points to $0$ and $\infty$.  The factor of $xy$ in $\Phi_1^*$ reflects these two fixed points for every value of $a$ and $b$, and the third factor provides the third fixed point.

\begin{theorem} \label{irred}
$M_2(N)$ is geometrically irreducible for every $N>1$.
\end{theorem}

\begin{proof}
Consider the map
\begin{equation}\label{cdmdN}
\begin{CD}
\left\{ (x,y) \in \A^2 \colon \Phi_N^*\left(x,y,a,b\right)=0\right\} @>>> M_2(N) \smallsetminus \{\nu_2(N) \text{ points over } (3,3)\}\\
@VVV @VVV\\
\A^2 \smallsetminus \{ab = 1\} @>>> M_2\smallsetminus \{(3,3)\}\\
\end{CD}
\end{equation}
Each map is surjective, with the horizontal maps generically $6$-to$1$ as described previously, and the vertical maps $\nu_2(N)$-to-$1$.

The top map gives a surjection from the geometrically irreducible variety $\Phi_N^*\left(x,y,a,b\right)=0 $ to the variety $M_2(N) \smallsetminus \{\text{a finite set of points}\}$.  Therefore, $M_2(N)$ must also be irreducible.
\end{proof}

Let $C \subseteq M_2$ be an algebraic curve, and let 
 $Y_1(C,N) = p^{-1}_N(C)$.  

\begin{center}
\begin{tabular}{c c c}
$Y_1(C,N)$ & $\subseteq$ & $ M_2(N)$\\ 
$\downarrow$ & & $\downarrow$ \\ 
$C$ & $\subseteq$ & $ M_2$\\
\end{tabular}
\end{center}

\begin{cor}\label{irredcurves}
Consider the space of curves $C$ in $M_2$ of degree~$d$, each curve $C$ corresponding to a family of degree-$2$ rational maps $\phi_C$.
For every $N>1$, only a Zariski-closed subset of such curves have reducible modular curves $Y_1(C,N)$. 
\end{cor}

\begin{proof}
One version of Bertini's Theorem (see, for example,~\cite{egav}) states that if $X$ is geometrically irreducible and 
$\dim X \geq 2$, then the generic hyperplane section has the same property.
Using the Veronese embedding, we can generalize the statement above to degree-$d$ hypersurfaces.  In other words, the intersection of the irreducible variety $X$ with a generic hypersurface is irreducible.

A modular curve $Y_1(C,N)$ corresponds to the intersection of the  irreducible quasi-projective variety $M_2(N)$ and the hypersurface $\PP^1 \times p_N^{-1}(C) $. 
\end{proof}

\section{Maps with automorphisms}\label{autmapsprelim}

We continue using the notation from Sections~\ref{makemd} and~\ref{qpvir}.  In addition, we will use the following:

\begin{tabular}{l l }
$\mathcal O_{\phi}(\alpha) = \{ \phi^N(\alpha) : N\geq 0 \}$ & the forward orbit of a point $\alpha$ under $\phi$.\\
$\Fix(\phi) = \{ \alpha \in \PP^1 : \phi(\alpha) = \alpha \}$ &the set of fixed points of a map $\phi$. \\
\end{tabular}

 As before, take $\phi(x,y)\colon \PP^1 \to \PP^1$.
 Also, choose some $h \in \Aut(\phi)$ of prime order~$p$.
  Fix some $Q \in \PP^1$.  Since $h$ has finite order, $\mathcal O_h(Q) = \{Q_0, \ldots, Q_{p-1}\}$ is a finite set.  As a convention, we write 
$$
Q=Q_0\overset{h}{\mapsto} Q_1 \overset{h}{\mapsto} \cdots \overset{h}{\mapsto} Q_{p-1} \overset{h}{\mapsto} Q_0=Q.
$$
Unless $Q \in \Fix(h)$, the $Q_i$ are distinct, since $h$ has prime order.
If for some $i$  we have $\infty = [1:0] = Q_i$, choose $f\in \PGL_2$ so that $f$ interchanges $Q_i$ and a point not on the orbit of $Q$.   Replacing $\phi$ by $\phi^f$ and $h$ by $h^f$, we may assume that none of the $Q_i$ is infinity. 

\begin{lemma}\label{2fp}\label{fplem}
Let $\phi\colon \PP^1 \to \PP^1$ and $h \in \Aut(\phi)$ of prime order~$p$.
\begin{enumerate}[\textup(a\textup)]
\item
If the characteristic of $K$ is different from $p$, then $\Fix(h)$ consists of exactly two distinct points.

\item
If $Q\in \Fix(h)$ then $\left| \mathcal O_{\phi}(Q)\right| \leq 2$.
\end{enumerate}
\end{lemma}

\begin{proof}
  Any nontrivial element of $\PGL_2$ has exactly two fixed points, counted with multiplicity.  The only elements with exactly one fixed point are equivalent under a change of coordinates to a nontrivial translation (where the fixed point is the point at infinity).  If $\ch K \neq p$, then none of these has order $p$, so $h$ has exactly two fixed points.

If $h(Q) = Q$ then for any $k\geq 0$, $h\phi^k(Q) =\phi^kh(Q) =\phi^k(Q)$ because $h\in\Aut(\phi)$.  In other words, every point in the orbit of $Q$ is fixed by $h$, so there can be at most two distinct points on the orbit.
\end{proof}

\begin{define}
Let $\phi\colon \PP^1 \to \PP^1$ and $h \in \Aut(\phi)$.  A point $Q \in \PP^1$ has \emph{$h$-period $N$ for $\phi$}, if $\phi^N(Q) = h^j(Q)$ for some $j>0$ and $p \nmid j$.  A point
$Q \in \PP^1$ has  \emph{primitive $h$-period $N$ for $\phi$}, if $N>0$ is the smallest such integer.
\end{define}

If $Q$ has $h$-period $N$, then $\phi^{pN}(Q) = h^{jp}(Q) = Q$ since $h$ has order $p$.  In other words, $Q$ is a point of order $pN$ for $\phi$, but one with the additional special property that $h^j(Q)$ is on the orbit $\mathcal O_{\phi}(Q)$.
In analogy with  the dynatomic polynomials $\Phi_N^*$,  we wish to create polynomials that have as roots  points of primitive $h$-period $N$ for $\phi$.   

Let $\phi(x,y) \colon \PP^1 \to \PP^1$ be a rational map.  Suppose that $h \in \Aut(\phi)$ and that $\Fix(h) = \{P_i\}$ with $P_i=[x_i:y_i]$.  By abuse of notation, we  will write $\phi^N(x,y) - h^j(x,y)$ to represent the polynomial whose roots are exactly the points $Q \in \PP^1$ such that $\phi^N(Q) = h^j(Q)$.  For example, if we have 
$$
\phi(x,y) = [F_1(x,y) : G_1(x,y)]  \quad \text{and} \quad h(x,y) = [ax+by:cx+dy],
$$
 then by 
$\phi(x,y) - h(x,y)$ we mean the polynomial
$$(cx+dy)F_1(x,y) - (ax+by)G_1(x,y).$$

\begin{define}\label{psidef}
\begin{align}
\Psi_{pN,\phi,h}(x,y) 
&= \prod_{j=1}^{p-1} \left(\phi^N(x,y) - h^j(x,y)\right)\\
\Psi_{pN,\phi,h}^*(x,y) 
&= \prod_{\substack{k\mid N \\
 pk\nmid N}}\left(\Psi_{pk,\phi,h}(x,y)\right)^{\mu(N/k)}\\
\widetilde{\Psi_{pN,\phi,h}^*}(x,y)&=\frac{\Psi_{pN,\phi,h}^*(x,y) }{\prod_{P_i\in \Fix(h)}(y_ix-x_iy)^{\delta_i}}\\
\text{where } \delta_i & =
\ord_{(y_ix-x_iy)}\Psi_{pN,\phi,h}^*(x,y).\nonumber
 \end{align}
We call $\widetilde{\Psi_{pN,\phi,h}^*}(x,y)$ the \emph{$h$-tuned dynatomic polynomial for $\phi$}.  \textup(The justification for the term ``polynomial'' in this definition comes from Proposition~\ref{polyprop}.\textup)
\end{define}

When the rational map $\phi$ and the automorphism $h$ are fixed, we may suppress dependence on $\phi$ and $h$ in our notation, writing simply $\Psi_{pN}$, $\Psi_{pN}^*$, and $\widetilde{\Psi_{pN}^*}$.  
 By construction, the roots of $\Psi_{pN}$ are the points $Q=[x:y]$ such that $\phi^N(Q)=h^j(Q)$ for some $1\leq j \leq p-1$.  With $\Psi_{pN}^*$, we are eliminating (usually) the points $Q$ such that $\phi^k(Q) = h^j(Q)$ for some $k<N$.     The need to eliminate the fixed points of $h$ as well will become clear as we proceed. 

\begin{example}\label{psiegs}
Let $\phi(x,y) = [x^2 - 2xy:-2xy+y^2]$.  We see that  $f(x,y) = [y:x]$ is  an automorphism of order~$2$   and $g(x,y)=[x-y:x]$ is an automorphism of order~$3$ for $\phi$ .  We first compute a few of the dynatomic polynomials $\Phi^*_{N,\phi}(x,y)$.  (The following were computed with Mathematica.)

\begin{align*}
\Phi_2^*(x,y) 
& =-x^2+y x-y^2\displaybreak[0]\\
\Phi_3^*(x,y) 
& = 3 \left(x^3-3 y x^2+y^3\right) \left(x^3-3 y^2
   x+y^3\right)\displaybreak[0]\\
\Phi_4^*(x,y) 
& = \left(-5 x^4+10 y x^3-5 y^3 x+y^4\right)\\
& \qquad \cdot  \left(x^4+y x^3-9   y^2 x^2+y^3 x+y^4\right) \left(-x^4+5 y x^3-10 y^3 x+5
   y^4\right)\displaybreak[0]\\
\Phi_6^*(x,y) 
& = \left(7 x^6-21 y x^5+35 y^3 x^3-21 y^4 x^2+y^6\right)\\
& \quad \cdot\left(x^6-21 y^2 x^4+35 y^3 x^3-21 y^5 x+7  y^6\right)\\
   & \quad \cdot  \left(x^6-6 y x^5-6 y^2 x^4+29 y^3 x^3-6 y^4 x^2-6 y^5x+y^6\right) \\
    &\quad \cdot \left(x^{18}-18 y x^{17}-54 y^2 x^{16}+1167
   y^3 x^{15}-2466 y^4 x^{14} -7344 y^5 x^{13} \right. \\
   & \qquad +31065 y^6
   x^{12}-20619 y^7 x^{11}-54513 y^8 x^{10}+99326 y^9x^9\\
   & \quad\qquad  -34119 y^{10} x^8-47844 y^{11} x^7+51072 y^{12}
   x^6-16155 y^{13} x^5\\
   & \qquad \qquad  \left. -621 y^{14} x^4+1329 y^{15} x^3-207
   y^{16} x^2+y^{18}\right)\\
   & \qquad \cdot  \left(x^{18}-207 y^2
   x^{16}+1329 y^3 x^{15}-621 y^4 x^{14}-16155 y^5
   x^{13}\right. \\
   &\qquad \quad +51072 y^6 x^{12}-47844 y^7 x^{11} -34119 y^8
   x^{10}+99326 y^9 x^9\\
   & \qquad \qquad -54513 y^{10} x^8-20619 y^{11}
   x^7+31065 y^{12} x^6-7344 y^{13} x^5\\
   & \qquad \qquad \quad \left. -2466 y^{14}
   x^4+1167 y^{15} x^3-54 y^{16} x^2 -18 y^{17}
   x+y^{18}\right)
\end{align*}
   
Note that 
$$
\Fix(f) = \{\pm 1\} \quad \text{and}
 \quad \Fix(g) = \{ \text{roots of } z^2-z+1 \}.
 $$
 In other words, points in $\Fix(g)$ are the primitive sixth roots of unity.

Since $f$ has order~$2$, we compute the first few $f$-tuned dynatomic polynomials $\widetilde{\Psi^*_{2N,\phi,f}}$.
\begin{align*}
\Psi_{2,\phi,f} (x,y) &= \Psi_{2,\phi,f}^*(x,y)\\
&= x \left(x^2-2 x y\right)-y \left(y^2-2 x y\right) = (x-y) \left(x^2-y x+y^2\right)\\
\widetilde{\Psi_{2,\phi,f}^*}(x,y) &= x^2-y x+y^2\displaybreak[0]\\
\displaybreak[0]\\
\Psi_{4,\phi,f} (x,y) &= \Psi_{4,\phi,f}^*(x,y) \\
&= x\left(x^2-2 x y\right)^2-2 \left(x^2-2 x y\right)
   \left(y^2-2 x y\right) - y\left(y^2-2 x y\right)^2\\
   & \qquad \qquad -2 \left(x^2-2 x y\right)
   \left(y^2-2 x y\right)\\
   &= (x-y) \left(x^4+y x^3-9 y^2 x^2+y^3 x+y^4\right)\\
\widetilde{\Psi_{4,\phi,f}^*}(x,y) &=   x^4+y x^3-9 y^2 x^2+y^3 x+y^4\displaybreak[0]\\
\displaybreak[0]\\
\Psi_{6,\phi,f} (x,y) &= (x-y) \left(x^2-y x+y^2\right) \\
& \quad \cdot \left(x^6-6 y x^5-6 y^2 x^4+29 y^3 x^3-6 y^4 x^2-6 y^5 x+y^6\right)\\
\Psi_{6,\phi,f}^*(x,y)&=\frac{\Psi_{6,\phi,f} (x,y) }{\Psi_{2,\phi,f} (x,y)} = x^6-6 y x^5-6 y^2
   x^4+29 y^3 x^3-6 y^4 x^2-6 y^5 x+y^6\\
\widetilde{\Psi_{6,\phi,f}^*}(x,y) &= \Psi_{6,\phi,f}^*(x,y) = x^6-6 y x^5-6 y^2
   x^4+29 y^3 x^3-6 y^4 x^2-6 y^5 x+y^6
\displaybreak[0]\\
\displaybreak[0]\\
\intertext{Since $g$ has order~$3$, we compute the first few $g$-tuned dynatomic polynomials $\widetilde{\Psi^*_{3N,\phi,g}}$.}
\Psi_{3,\phi,g} (x,y) 
&=  \left(x \left(x^2-2 x y\right)-(x-y) \left(y^2-2 x y\right)\right)\\
& \qquad \cdot \left((y-x) \left(x^2-2 x y\right)-y \left(y^2-2 x y\right) \right)\\
&= -\left(x^3-3 y x^2+y^3\right) \left(x^3-3 y^2 x+y^3\right)\\
 \widetilde{\Psi_{3,\phi,g}^*}(x,y) 
 &= \Psi_{3,\phi,g}^*(x,y) = \Psi_{3,\phi,g}(x,y) \\
 &= -\left(x^3-3 y x^2+y^3\right) \left(x^3-3 y^2 x+y^3\right) \displaybreak[0]\\
 \displaybreak[0]\\
 \Psi_{6,\phi,g}(x,y) & =-\left(x^2-y x+y^2\right)^2 \left(x^3-3 y x^2+y^3\right)
   \left(x^3-3 y^2 x+y^3\right)  \\
  \Psi_{6,\phi,g}^*(x,y) &= \frac{ \Psi_{6,\phi,g}(x,y)}{\Psi_{3,\phi,g} (x,y) } = \left(x^2-y x+y^2\right)^2 \\
\widetilde{\Psi_{6,\phi,g}^*}(x,y) &=     1 
\end{align*}

\end{example}

\medskip

Based on this example, one might conjecture that $\widetilde{\Psi_{pN}^*}$ is always a polynomial, and that it divides $\Phi_{pN}^*$.  In Section~\ref{red}, we prove these assertions.   After determining that $\widetilde{\Psi_{pN}^*}$ is almost always nontrivial, we will be able to conclude that the dynatomic polynomials $\Phi_{pN}^*$ are reducible for infinitely many $N$, and hence so are the moduli spaces $M_d(pN,\mathfrak C_p)$.

\section{Basic properties of $h$-tuned dynatomic polynomials}

Throughout, we fix a rational map $\phi\colon \PP^1 \to \PP^1$ of degree $d \geq 2$, and a map $h \in \Aut(\phi)$ of prime order~$p$.

\begin{lemma}\label{msmall}
Suppose $Q$ has $h$-period $N$ for $\phi$. 
\begin{enumerate}[\textup(a\textup)]
\item\label{eqorb}
Then 
$ \phi^{kN}(Q) \in \mathcal O_h(Q)$ for all $k$.

\item
If $Q$ has primitive $h$-period $m$, then  $m \mid N$.
\end{enumerate}
\end{lemma}

\begin{proof}
Suppose  $\phi^N(Q) = h^j(Q) \neq Q$ for some $1 \leq j \leq p-1$.  As sets, $\mathcal O_h(Q) = \mathcal O_{h^j}(Q)$, so we may rename the automorphism so that $\phi^N(Q) = h(Q)$.  We now show, in fact, that $\phi^{kN}(Q)=h^k(Q)$ for all $k$.  Since $h \in \Aut(\phi)$, we know that $\phi^N h = h\phi^N$ for all $N$.  Then,
\begin{align*}
\phi^{kN}(Q) &= \underbrace{\phi^N\circ \phi^N\circ \cdots \circ\phi^N}_{k\text{ copies}}(Q)\\
&= h^k(Q) \qquad \text{using the facts that $\phi^N h = h\phi^N$ and $\phi^N(Q) = h(Q)$.}  \displaybreak[0] \\
\end{align*}
For the second result, we have
\begin{align*}
h(Q) &= \phi^{N}(Q) = \phi^{qm+r}(Q) && \text{with }0\leq r<m\\
&= \phi^{r}\phi^{qm}(Q) = \phi^{r}h^k(Q) && \text{for some $k$, by the result above}\\
&= h^k \phi^{r}(Q) && \text{since } h \in \Aut(\phi).\\
\end{align*}
So $\phi^r(Q) = h^{1-k}(Q) \in \mathcal O_h(Q)$.
By minimality of $m$, $r=0$ and so $m \mid N$.
\end{proof}

If $Q\neq\infty$ and $\phi(Q) \neq \infty$, then we may expand $\phi$ around $Q$: 
$$
\phi(z) =  \sum_{i=0}^n \lambda_i(\phi,Q)(z-Q)^i + {\rm O}\left((z-Q)^{n+1}\right),
$$
where  ${\rm O}\left((z-Q)^{n+1}\right)$ represents a function vanishing to order at least $n+1$ at $z=Q$.  We take this as the definition of the $ \lambda_i(\phi,Q)$.

\begin{remark}
Note that $\lambda_0(\phi,Q) = \phi(Q)$.  Furthermore, if $Q$ is periodic with period~$N$, then $\lambda_1(\phi^N,Q)$ is the multiplier of the cycle  as defined on page~\pageref{multdef}.  Also,  expanding both $\phi^N(z)$ and $h^j(z)$ around $Q$ as described, we see that the following two conditions are equivalent:
\begin{enumerate}
\item
$\lambda_i\left(\phi^N,Q\right) = \lambda_i(h^j,Q)$ for $0\leq i \leq n$, and
\item
$\ord_{z=Q}\left(\Psi_{pN,\phi,h}\right)\geq n+1$.
\end{enumerate}
The $\lambda_i$ can therefore be used to compute the order of vanishing of the $h$-tuned dynatomic polynomials at a point $Q \in \PP^1$, which is the key to proving they are indeed polynomials.

We recall the following definitions from complex dynamical systems.  The cycle containing $Q$ is
\begin{itemize}
\item
\emph{attracting} if $\left| \lambda_1(\phi^N,Q) \right| <1$,
\item
\emph{repelling} if $\left| \lambda_1(\phi^N,Q) \right| >1$, and
\item
\emph{indifferent} (or neutral) if $\left| \lambda_1(\phi^N,Q) \right| =1$.
\end{itemize}
In the last case, if $ \lambda_1(\phi^N,Q) $ is a root of unity then the cycle is said to be \emph{rationally indifferent}.
\end{remark}

\begin{lemma}\label{equalcoeffs}
Let $f,g,h$ be rational  maps.  Then:
\begin{enumerate}[\textup(a\textup)]
\item
If $\lambda_i(f,h(Q)) = \lambda_i(g,h(Q))$ for all $0 \leq i \leq n$, then also
$\lambda_i(fh,Q) = \lambda_i(gh,Q)$ for all $0 \leq i \leq n$.
\item
If $\lambda_i(f,Q) = \lambda_i(g,Q)$ for all $0 \leq i \leq n$ \textup(in particular, we require $f(Q)=g(Q)$\textup), then also
$\lambda_i(hf,Q) = \lambda_i(hg,Q)$ for all $0 \leq i \leq n$.
\end{enumerate}
\end{lemma}

\begin{remark}
Suppose we have $\phi, h \in K(z)$ and let $f\in \PGL_2$ be some change of coordinates.  Then Lemma~\ref{equalcoeffs} says that
\begin{align*}
\lambda_i(\phi,Q) &= \lambda_i(h,Q) \text{ for all } 0\leq i\leq n \\
&\Longrightarrow
\lambda_i(\phi f,f^{-1}Q) = \lambda_i(h f, f^{-1}Q) \text{ for all } 0\leq i\leq n\\
&\Longrightarrow
\lambda_i(f^{-1}\phi f , f^{-1}Q) = \lambda_i(f^{-1}h f , f^{-1}Q) \text{ for all } 0\leq i\leq n .
\end{align*}
So if $\phi(Q) = h(Q)$,  then equality of the first $n$ coefficients, $\lambda_i(\phi,Q)$ and $\lambda_i(h,Q)$, is preserved under $\PGL_2$ conjugation.  This is how the Lemma will be applied.
\end{remark}

\begin{proof}
The proof is essentially the chain rule.  The $n=0$ case is clear.   For $1\leq m\leq n$,
\begin{align*}
\lambda_m(fh, P) &=
\sum_{k=1}^m  \lambda_k(f,h(Q) )
\prod_{\substack{k\text{-tuples } (i_1, \ldots, i_k)\\ i_1+\cdots+i_k=m}} \lambda_{i_j}(h,Q)\displaybreak[0]\\
&=
\sum_{k=1}^m  \lambda_k(g,h(Q) )
\prod_{\substack{k\text{-tuples } (i_1, \ldots, i_k)\\ i_1+\cdots+i_k=m}} \lambda_{i_j}(h,Q)
&&\text{by the hypothesis}\displaybreak[0]\\
&=
\lambda_m(gh, Q) \displaybreak[2]\\
\lambda_m(hf, Q) &=
\sum_{k=1}^m  \lambda_k(h,f(Q) )
\prod_{\substack{k\text{-tuples } (i_1, \ldots, i_k)\\ i_1+\cdots+i_k=m}} \lambda_{i_j}(f,Q)\displaybreak[0]\\
&=
\sum_{k=1}^m  \lambda_k(h,g(Q) )
\prod_{\substack{k\text{-tuples } (i_1, \ldots, i_k)\\ i_1+\cdots+i_k=m}} \lambda_{i_j}(g,Q)
&&\text{by the hypothesis}\displaybreak[0]\\
&=
\lambda_m(hg, Q). &&\qedhere
\end{align*}

\end{proof}

\begin{lemma}\label{derivrel1}
 Let $Q_i = h^i(Q)$ as before. For any $Q\in \overline K$, \begin{equation}
 \prod_{i=0}^{p-1}\lambda_1(h,Q_i) = 1. \label{hprod1}
 \end{equation}
Furthermore, if $Q \in \Fix(h)$ and $K$ has characteristic different from $p$, then $\lambda_1(h,Q)$ is a primitive $p\tha$ root of unity. 
\end{lemma}

\begin{proof}
The first result follows from the fact that $h^p(z) = z$.  For all $Q \in \overline K$, this gives $\left(h^p\right)'(Q) = 1$, so the numbers $\lambda_1(h,Q)=h'(Q)$ are well-defined (and nonzero) for all $Q \in \overline K$.  So then
\begin{eqnarray*}
 \prod_{i=0}^{p-1}\lambda_1(h,Q_i) &=&  \prod_{i=0}^{p-1} h'(Q_i)\\
&=&(h^p)'(Q) = 1.
\end{eqnarray*}

If $Q \in \Fix(h)$, equation~\eqref{hprod1} becomes
$\left( \lambda_1(h,Q)\right)^p =1$.  If $\lambda_1(h, Q) = 1$, then $Q$ is a double root of $h(z) - z = 0$.  The automorphism $h$ has exactly two fixed points counted with multiplicity, and from Lemma~\ref{2fp} we know that they are distinct since $\ch K \neq p$.  So $\lambda_1(h, Q)$ must be a primitive $p\tha $ root of unity, since $p$ is prime.
\end{proof}

\begin{lemma}\label{derivrel2}
Suppose that $\phi(Q) = h(Q)$.  Then
 \begin{align}
\frac{\lambda_1(\phi,Q_i)}{\lambda_1(h,Q_i)} &= \frac{\lambda_1(\phi,Q_j)}{\lambda_1(h,Q_j)}
\quad \text{for all $i, j \geq 0$.}\label{lambdaseq} \\
\prod_{i=0}^{p-1}\lambda_1(\phi, Q_i) &=\left( \frac{\lambda_1(\phi,Q)}{\lambda_1(h, Q)}\right)^p. \label{phiprod}
 \end{align}

\end{lemma}

\begin{proof}
As described in Section~\ref{autmapsprelim}, we assume without loss of generality that infinity is not in $\mathcal O_h(Q)$.  The required change of coordinates is permitted by Lemma~\ref{equalcoeffs}.

Since $\phi h = h\phi$,
\begin{eqnarray}
\lambda_1(\phi h, Q_i) &=&  \lambda_1(h \phi, Q_i), \nonumber\\
\lambda_1(\phi, Q_{i+1}) \lambda_1(h, Q_i) &=& \lambda_1(h, Q_{i+1}) \lambda_1(\phi, Q_i). \label{lambda1} 
\end{eqnarray}
Dividing each side of equation~\eqref{lambda1} by the product $ \lambda_1(h, Q_i)  \lambda_1(h, Q_{i+1}) $ --- which is nonzero by Lemma~\ref{derivrel1} --- along with an induction gives the first result. 

To prove the second result, divide each side of equation~\eqref{lambda1} by $\lambda_1(h, Q_{i})$ to get
\begin{eqnarray}
\lambda_1(\phi, Q_{i+1})&=&\lambda_1(\phi, Q_{i})\left(\frac{\lambda_1(h, Q_{i+1})}{\lambda_1(h, Q_{i})}\right). \label{qi+1}
\end{eqnarray}
Repeatedly using the substitution in equation~\eqref{qi+1}, we have 
\[
\lambda_1(\phi , Q_i) = \lambda_1(\phi, Q_{0})\left(\frac{\lambda_1(h, Q_{i})}{\lambda_1(h, Q_0)}\right).
\]
  Now, applying the identity  in equation~\eqref{hprod1}, we have
\begin{align*}
\prod_{i=0}^{p-1}\lambda_1(\phi, Q_i) &=  \prod_{i=0}^{p-1}\lambda_1(\phi, Q_0)\left(\frac{\lambda_1(h, Q_{i})}{\lambda_1(h, Q_{0})}\right)\\
&=\left( \frac{\lambda_1(\phi,Q )}{\lambda_1(h, Q)}\right)^p. &&\qedhere
\end{align*}
\end{proof}

\begin{lemma}\label{creqns}
Assume $\phi(Q) = h(Q)$.  If  $\lambda_j(\phi,Q) = \lambda_j(h,Q)$ for all $1 \leq j \leq e$, then
\begin{enumerate}[\textup(a\textup)]
\item
 $\lambda_j(\phi,Q_{i}) = \lambda_j(h,Q_{i})$ for all $Q_i\in \mathcal O_h(Q)$ and all $1\leq j \leq e$, and
 \item
 $ \lambda_j(\phi^i,Q) =  \lambda_j(h^i,Q)$ for all $i$  and all $1 \leq j \leq e$.  In particular, we have
\begin{itemize}
\item
$\lambda_1(\phi^p,Q)= 1$, and 
\item
$ \lambda_j(\phi^p,Q) = 0$ for all $2 \leq j \leq e$.
\end{itemize}
\end{enumerate}
\end{lemma}

\begin{proof}
Note first that $\lambda_1(h,Q_i)\neq 0$ by the fact that $\prod_{i=1}^p\lambda_1(h,Q_i) =1$ from Lemma~\ref{derivrel1}.

The first assertion is proved by induction.  
If $\lambda_1(h,Q_i) = \lambda_1(\phi,Q_i)$, neither is~$0$, so we may cancel them both in equation~\eqref{lambda1} to get $\lambda_1(h,Q_{i+1}) = \lambda_1(\phi,Q_{i+1})$.

Now suppose the implication holds for $1\leq j \leq e-1$, and that  $\lambda_j(\phi,Q_i) = \lambda_j(h,Q_i)$ for all $1 \leq j \leq e$.
Because $\phi h = h \phi$, we get
\begin{align*}
\lambda_e(\phi h, Q_i) 
&=
\lambda_e(h \phi , Q_i). 
\end{align*}
This gives
\begin{equation*}
\begin{split}
\lambda_e(\phi, Q_{i+1})\left(\lambda_1(h, Q_i)\right)^e &+ (S_1) + \lambda_1(\phi, Q_{i+1}) \lambda_e(h, Q_i) \\
&=
\lambda_e(h, Q_{i+1})\left(\lambda_1(\phi, Q_i)\right)^e + (S_2) + \lambda_1(h, Q_{i+1}) \lambda_e(\phi, Q_i),
\end{split}
\end{equation*}
where $(S_1)$ represents terms involving $\lambda_j(\phi,Q_{i+1})$ and $\lambda_k(h,Q_i)$ with $1\leq j\leq e-1$ and $k < e$, and similarly for $(S_2)$.  These terms will be equal on each side by the induction hypothesis, so they cancel, as do the final two terms by the fact that $\lambda_e(h , Q_i) = \lambda_e( \phi, Q_i)$.  We are left with
\begin{eqnarray*}
\lambda_e(\phi, Q_{i+1})\left(\lambda_1(h, Q_i)\right)^e 
&=&
\lambda_e(h, Q_{i+1})\left(\lambda_1(\phi, Q_i)\right)^e\\
\lambda_e(\phi, Q_{i+1})
&=&
\lambda_e(h, Q_{i+1}),
\end{eqnarray*}
where the last equality follows from the fact that $\lambda_1(\phi, Q_i)
=\lambda_1(h, Q_i) \neq 0$ again.

For the second assertion, note that  $\lambda_j(\phi^i,Q)$ is some polynomial combination of $\lambda_k(\phi,Q_i)$ for $1\leq k \leq j$ and $Q_i\in \mathcal O_h(Q)$, and   $\lambda_j(h^i,Q)$ is the exact same polynomial combination of $\lambda_k(h,Q_i)$.  By the first assertion, then, these two must be equal for $1 \leq j \leq e$.

The final two statements follow immediately from the above and the fact that $h^p(z) = z$.  From which we know that  $\lambda_1(h^p,z) = 1$ and $\lambda_j(h^p,z) = 0$ for $j\neq1$.  In particular, these hold at $z=Q$.
\end{proof}

\begin{lemma}\label{phij-hj}
Assume $\phi(Q) = h(Q)$, and
let $e$ be the smallest positive integer such that $\lambda_e(\phi,Q) \neq \lambda_e(h,Q)$.  Then
$$\displaystyle \lambda_e(\phi^j, Q) - \lambda_e(h^j,Q)= j\left( \prod_{i=1}^{j-1}\lambda_1(\phi,Q_i)\right) \left(\lambda_e(\phi,Q)-\lambda_e(h,Q)\right).$$
\end{lemma}

\begin{proof}
We proceed by induction.  The claim is trivial for $j=1$.  Assume the relation holds for $j-1$.  Note that by Lemma~\ref{creqns}, for all $0\leq f<e$ we have $\lambda_f(\phi^i,Q_j)= \lambda_f(h^i,Q_j)$ for all non-negative integers $i$ and $j$, which gives equality of the terms marked $(S_1)$ and $(S_2)$ in the first equation below.
\begin{align}
\lambda_e(\phi^j,Q) &- \lambda_e(h^j,Q)\nonumber\\
&=
\left( \lambda_e(\phi,Q_{j-1})\lambda_1(\phi^{j-1},Q)^e + (S_1) + \lambda_1(\phi,Q_{j-1})\lambda_e(\phi^{j-1},Q) \right) \nonumber\\
& \qquad - \left( \lambda_e(h,Q_{j-1})\lambda_1(h^{j-1},Q)^e + (S_2) + \lambda_1(h,Q_{j-1})\lambda_e(h^{j-1},Q) \right) \nonumber\\
&=
\lambda_1(\phi^{j-1},Q_1)^e (\lambda_e(\phi,Q_j)-\lambda_e(h,Q_j)) \nonumber\\
& \qquad + \lambda_1(\phi,Q_{j-1}) (\lambda_e(\phi^{j-1},Q)-\lambda_e(h^{j-1},Q))\nonumber\\
&=
\lambda_1(\phi^{j-1},Q)^e (\lambda_e(\phi,Q_{j-1})-\lambda_e(h,Q_{j-1})) \nonumber \\
& \qquad +(j-1)\left(\prod_{i=1}^{j-1} \lambda_1(\phi,Q_i)\right)  \left(\lambda_e(\phi,Q)-\lambda_e(h,Q)\right). \label{jinduct}
\end{align}

It remains to calculate the first term in this sum.  First, note that 
\begin{align*}
\lambda_1(\phi^{j-1},Q) &= \prod_{i=0}^{j-2}\lambda_1(\phi,Q_i)
&&\text{(recall that $Q = Q_0$).}
\end{align*}
We again use the fact that $\phi h = h\phi$.  Equality of the terms $(S_1)$ and $(S_2)$ follows as usual.

\begin{align*}
&\lambda_e(\phi h,Q_i) = \lambda_e(h\phi, Q_i)\\
&\lambda_e(\phi,Q_{i+1})\lambda_1(h,Q_i)^e + (S_1) + \lambda_1(\phi,Q_{i+1})\lambda_e(h,Q_i)\\
&\qquad \qquad = \lambda_e(h,Q_{i+1})\lambda_1(\phi,Q_i)^e + (S_2) + \lambda_1(h,Q_{i+1})\lambda_e(\phi,Q_i)\\
&\lambda_1(\phi,Q_i)^e\left(\lambda_e(\phi,Q_{i+1}) - \lambda_e(h,Q_{i+1})\right)
=
\lambda_1(\phi,Q_{i+1})\left(\lambda_e(\phi,Q_{i}) - \lambda_e(h,Q_{i})\right).
\end{align*}

So inductively again,
\begin{align*}
\lambda_1(\phi^{j-1},Q)^e (\lambda_e(\phi,Q_{j-1})&-\lambda_e(h,Q_{j-1}))\\
&= \left(\prod_{i=0}^{j-2} \lambda_1(\phi,Q_i)^e \right)
\left(\lambda_e(\phi,Q_{j-1}) - \lambda_e(h,Q_{j-1})\right)\\
&= 
\prod_{i=1}^{j-1}\lambda_1(\phi,Q_i)\left(\lambda_e(\phi,Q) - \lambda_e(h,Q)\right).
\end{align*}

Substituting this into equation~\eqref{jinduct} above gives the desired result.
\end{proof}

Proposition~$3.4$ in~\cite{dynunit} characterizes exactly when the order of vanishing of the dynatomic polynomials is positive at some point.  We begin with a definition.

\begin{define}\label{apdef2}
Let 
\begin{align*}
a_Q(N) &= \ord_{z=Q}(\Phi_{N}) = \ord_{z=Q}(\phi^N(z) -z), \text{ and }\\
a_Q^*(N)& = \ord_{z=Q}(\Phi_{N}^*)  = \sum_{k\mid N}\mu(N/k)a_Q(k).\\
\end{align*}
\end{define}
Note that this is the same $a_Q(N)$ given by the intersection multiplicity on page~\pageref{zn}.

\begin{lemma}[Lemma 3.4 in~\cite{dynunit}]
Let $K$ be a field, let $X/K$ be a smooth projective curve, and let $\phi\colon X \to X$ be a non-constant morphism defined over $K$.  Suppose that $Q \in X$ is a fixed point of $\phi$. 
\begin{enumerate}[\textup(a\textup)]
\item
$a_Q(N) \geq a_Q(1)$ for all $N \geq 1$.
\item
$a_Q(N)> a_Q(1)$ if and only if one of the following two conditions is true:
\begin{enumerate}[\textup(i\textup)]
\item
$a_Q(1) = 1$  and $\phi'(Q)^N = 1$.
\item
$a_Q(1) \geq 2$ and $N=0$ in $K$, in which case $a_Q(N) \geq 2a_Q(1) - 1.$
\end{enumerate}
\end{enumerate}
\end{lemma}

We prove the analogous result for the $h$-tuned dynatomic polynomials, beginning again with a definition.
\begin{define}\label{apdef}
For $p$ a prime, let 
\begin{align*}
b_Q(pN) &= \ord_{z=Q}(\Psi_{pN}) = \sum_{j=1}^{p-1}  \ord_{z=Q}\left(\phi^N(z) - h^j(z)\right)\displaybreak[0] \\
b_Q^*(pN)& = \ord_{z=Q}(\Psi_{pN}^*)  = \sum_{\substack{k\mid N\\ pk \nmid N}}\mu(N/k)b_Q(pk)\displaybreak[0]\\
  \widetilde{b_Q^*}(pN) & = \ord_{z=Q}\left(\widetilde{\Psi_{pN}^*}\right)\\
  &=
\begin{cases}
0 & \text{if } Q\in \Fix(h) \\
b_Q^*(pN) & \text{otherwise.}
\end{cases}
\end{align*}
\end{define}

\begin{lemma}\label{a1}\label{an}
 Suppose $Q$ has $h$-period~$1$ for $\phi$ and $Q \notin \Fix(h)$.  If $p \mid N$, then $b_Q(pN) = 0$.  If $p \nmid N$, then
\begin{align}
\left(\frac{\lambda_1(\phi,Q) }{\lambda_1(h^j, Q)}\right)^{N}\neq 1 \text{ for all $j$}
&\Longrightarrow b_Q(pN)=b_Q(p) = 1\\
\begin{matrix}
\left(\frac{\lambda_1(\phi,Q) }{\lambda_1(h^j, Q)}\right)^{N}= 1\text{ for some $j$}\\
\text{ and } \lambda_1(\phi,Q)\neq \lambda_1(h^j, Q)\\
\end{matrix} \Bigg\}
&\Longrightarrow b_Q(pN)>b_Q(p) = 1\\
\lambda_1(\phi, Q)= \lambda_1(h^j, Q) \text{ for some $j$}\
&\Longrightarrow
\begin{cases}
b_Q(pN) > b_Q(p)>1 & \text{if }\ch K \mid  N  \\
 b_Q(pN) = b_Q(p)>1 & \text{otherwise.}
 \end{cases}\label{eqlambdas}
\end{align}

\end{lemma}

\begin{proof}
Before continuing, we have some reductions.
From the definition, we have 
$$
\Psi_{pN,\phi,h} = \Psi_{pN,\phi,h^j}
$$
 for any $1\leq j \leq p-1$.  So by renaming the automorphism, we may assume that $\phi(Q) = h(Q) \neq Q$.   
By the proof of Lemma~\ref{msmall}, then,  $\phi^k(Q) = h^k(Q)$ for all $k$.  In other words, for all $Q_i\in \mathcal O_h(Q)$, $\phi(Q_i) = h(Q_i)$. 
If $Q$ is a point such that $\phi(Q) = h(Q)$, we may change coordinates so that $Q=0 =[0:1]$ and no $Q_i \in \mathcal O_h(Q)$  satisfies $Q_i = \infty = [1:0]$.  We will see that the condition 
$\left(\frac{\lambda_1(\phi,Q) }{\lambda_1(h, Q)}\right)^{N} = 1$
is equivalent to 
$\lambda_1(\phi^{N},Q) = \lambda_1(h^j, Q)$ for $N \equiv j \pmod p$.  By Lemma~\ref{equalcoeffs} each of the conditions above are preserved under this change.

Since $h$ has order $p$, every point $Q\in \PP^1$ has period $p$ under $h$.  For $Q \notin \Fix(h)$, the primitive period for $Q$ must be $p$.  Therefore, $Q$ has primitive period $p$ for $\phi$ as well.  Since $\phi^p(Q) =Q\neq h^j(Q)$ for any $1\leq j \leq p-1$, we see that $b_Q(pN) = 0$ whenever $p \mid N$.

Because $Q$ has primitive period $p$ for $h$, the orbit $\mathcal O_h(Q)$ consists of $p$ distinct points.   We assume that $\phi(Q) = h(Q)$, so $\phi(Q) \neq h^j(Q)$ for $2 \leq j \leq p-1$.  Hence, $b_Q(p) = \ord_{z=Q}\left(\phi(z) - h(z)\right)$.   We now proceed, focusing just on this term.

For each $i$, we have
\begin{eqnarray}
\phi(z) &=&  Q_{i+1} +  \sum_{j=1}^{n}\lambda_j(\phi,Q_i)( z-Q_i)^j + \mathcal{O}\left((z-Q_i)^{n+1}\right)\label{phiaround0} \\
h(z) &=&  Q_{i+1} +  \sum_{j=1}^{n}\lambda_j(h,Q_i)( z-Q_i)^j + \mathcal{O}\left((z-Q_i)^{n+1}\right)
\label{haround0}\\
\phi(z) - h(z) &=&  \sum_{j=1}^{n}\left(\lambda_j(\phi,Q_i)-\lambda_j(h,Q_i)\right)(z-Q_i)^j +  \mathcal{O}\left((z-Q_i)^{n+1}\right).\nonumber
\end{eqnarray}
From this, we see that $b_Q(p) = 1$ if $\lambda_1(\phi,Q) \neq \lambda_1(h,Q)$.  Otherwise $b_Q(p) = e >1$ where $e $ is the smallest integer such that $\lambda_e(\phi,Q)\neq \lambda_e(h,Q)$.  (Note  there must be such an $e$ since both are rational maps, but $\deg \phi >\deg h$.)

Assume now that  $\lambda_1(\phi,Q) \neq \lambda_1(h,Q)$, and  let $N=pM+j$ for some $1\leq j \leq p-1$.  Iterating equation~\eqref{phiaround0} starting with $Q_0 =0$, we have
\begin{align}
\phi(z) 
& =  Q_1 + \lambda_1(\phi, 0) z + \mathcal{O}\left(z^{2}\right) \label{phi1}\\
\phi^{p}(z) 
& = \left( \prod_{i=0}^{p-1} \lambda_1(\phi, Q_i)\right) z  + \mathcal{O}(z^{2}) 
\qquad \text{(recall that $Q_{p}=Q_0=0 $)} \displaybreak[0] \nonumber \\
&= \left( \frac{\lambda_1(\phi, 0)}{\lambda_1(h, 0)}\right)^p z + \mathcal{O}(z^{2})  
\qquad\qquad\text{by Lemma~\ref{derivrel2}.}\displaybreak[0]\label{phip}\\
\phi^{pM}(z)
 & =   \left( \frac{\lambda_1(\phi, 0)}{\lambda_1(h, 0)}\right)^{pM} z + \mathcal{O}(z^{2}).\displaybreak[0]\nonumber\\
\phi^{pM+j}(z) 
& = Q_{j} + \left( \frac{\lambda_1(\phi, 0)}{\lambda_1(h, 0)}\right)^{pM} \prod_{i=0}^{j-1}\lambda_1(\phi, Q_i) z + \mathcal{O}(z^{2}).\label{2n+1coeff}
\end{align}
Combining this with equation~\eqref{haround0}, we have
\begin{align*}
\phi^{pM+j}(z) -h^j(z) 
&=
\left( \left( \frac{\lambda_1(\phi, 0)}{\lambda_1(h, 0)}\right)^{pM} \prod_{i=0}^{j-1} \lambda_1(\phi, Q_i) - \lambda_1(h^j, 0)\right) z + \mathcal{O}(z^{2})\\
&=
\left( \left( \frac{\lambda_1(\phi, 0)}{\lambda_1(h, 0)}\right)^{pM} \prod_{i=0}^{j-1} \lambda_1(\phi, Q_i) - \prod_{i=0}^{j-1} \lambda_1(h, Q_i) \right) z + \mathcal{O}(z^{2}).
 \end{align*}
So $b_Q(pN) = b_Q(p) = 1$ unless  
the coefficient of $z$ above vanishes, or in other words unless
\begin{align}
\left( \frac{\lambda_1(\phi, 0)}{\lambda_1(h, 0)} \right)^{pM}
\prod_{i=0}^{j-1}\frac{ \lambda_1(\phi, Q_i)}{\lambda_1(h, Q_i)}
&=1
&&\text{(since $\prod_{i=0}^{j-1} \lambda_1(h, Q_i) \neq 0$ by Lemma~\ref{derivrel1}).}\nonumber\\
\intertext{Equivalently, $b_Q(pN) = b_Q(p) = 1$ unless}
\left( \frac{\lambda_1(\phi, 0)}{\lambda_1(h, 0)} \right)^{pM +j} &= 1 &&\text{(by equation~\eqref{lambdaseq}).}
\label{lam1phim}
\end{align}
The first two assertions follow from this.

We now consider the case that $\lambda_1(\phi, Q)= \lambda_1(h, Q) $.  We saw above that in this case $b_Q(1)= e >1$ where $e$ is the smallest positive integer such that $\lambda_e(\phi, Q) \neq \lambda_e(h, Q) $.    
 So by Lemma~\ref{creqns}, 
 $\lambda_1(\phi^p, 0)= 1$, and 
$\lambda_j(\phi^p, 0)= 0$ for all $2 \leq j < e$.  Therefore 
\begin{eqnarray}
\phi^p(z) & = & z + \lambda_e(\phi^p,0) z^e + \mathcal{O}(z^{e+1}).\label{betae1}
\end{eqnarray}
We may iterate $\phi^p$ in this simpler form to find that
\begin{eqnarray*}
\phi^{pM}(z) & = & z + M\lambda_e(\phi^p,0) z^e + \mathcal{O}(z^{e+1}).
\end{eqnarray*}
Composing this version of $\phi^{pM}(z)$ with the expansion of $\phi$ in equation~\eqref{phiaround0}, we find the following.
\begin{align*}
\phi^{pM +j}(z)
 & = Q_{j} + \lambda_1(\phi^j,0)\left(z + M\lambda_e(\phi^p,0) z^e + \mathcal{O}(z^{e+1})\right) \\
& \qquad+ \lambda_2(\phi^j,0)\left(z + M\lambda_e(\phi^p,0) z^e + \mathcal{O}(z^{e+1})\right)^2 + \cdots\\
&= Q_{j} + \sum_{i=1}^{e-1}\lambda_i(\phi^j,0) z^i +\left(\lambda_1(\phi^j,0) n\lambda_e(\phi^p,0) + \lambda_e(\phi^j,0)\right)z^e + \mathcal{O}(z^{e+1}).\\
\phi^{pM+j}(z) &-h^j(z) 
=  \sum_{i=1}^{e-1}\left( \lambda_i(\phi^j,0)  -  \lambda_i(h^j,0)\right)z^i \\
& \qquad \qquad +\left(\lambda_1(\phi^j,0) M\lambda_e(\phi^p,0) + \lambda_e(\phi^j,0)-\lambda_e(h^j,0)\right)z^e +\mathcal{O}(z^{e+1}).
\end{align*}
By Lemma~\ref{creqns}, the terms $\lambda_i(\phi^j,0)  -  \lambda_i(h^j,0)$ vanish, hence $b_Q(pN) \geq e$. 
By Lemma~\ref{phij-hj},
  \begin{align*}
   \lambda_e(\phi^p,0) & = \lambda_e(\phi^p,0)- \lambda_e(h^p,0)
   \qquad\text{since } e>1 \text{ means } \lambda_e(h^p,0)=0,\\
   &=  p\left( \prod_{i=0}^{p-1}\lambda_1(\phi,Q_i)\right) \left(\lambda_e(\phi,0)-\lambda_e(h,0)\right).
   \end{align*}
   So the coefficient of $z^e$ vanishes if and only if
$$
\left(Mp\lambda_1(\phi^j,0)\prod_{i=1}^{p-1}\lambda_1(\phi,Q_i) +
j\prod_{i=1}^{j-1}\lambda_1(\phi,Q_i)\right)
\left(\lambda_e(\phi,0)
-\lambda_e(h,0)\right) = 0.
$$
Now,  $\lambda_e(\phi,0)-\lambda_e(h,0)\neq 0 $ by our choice of $e$. Further, 
 $$
 \lambda_1(\phi^j,0)=\prod_{i=0}^{j-1}\lambda_1(\phi,Q_i).
 $$
By Lemma~\ref{derivrel2}, 
$
\prod_{i=1}^{j-1}\lambda_1(\phi,Q_i)\neq 0,
$
 so we may divide by it.     Thus, $b_Q(pN)>e$ if and only if
\begin{equation*}
Mp\prod_{i=0}^{p-1}\lambda_1(\phi,Q_i) + j 
= 0.
\end{equation*}
From Lemma~\ref{derivrel2}, we know that 
$\prod_{i=0}^{p-1}\lambda_1(\phi,Q_i)=1,$
so this says $N=Mp+j =0$.
That is, the coefficient of $z^e$ vanishes if and only if the characteristic of $K$ divides $N$, and in this case $b_Q(pN) >e$. 
\end{proof}

\begin{lemma}\label{bnqfix}
 Suppose $Q \in \Fix(h)\cap \Fix(\phi)$ and $K$ has characteristic different from $p$.  Then

\begin{align}
\left(\frac{\lambda_1(\phi,Q) }{\lambda_1(h^j, Q)}\right)^{N}\neq 1 \text{ for all $j$}
&\Longrightarrow b_Q(pN)=b_Q(p) = p-1\\
\begin{matrix}
\left(\frac{\lambda_1(\phi,Q) }{\lambda_1(h^j, Q)}\right)^{N}= 1\text{ for some $j$}\\
\text{ and } \lambda_1(\phi,Q)\neq \lambda_1(h^j, Q)\\
\end{matrix} \Bigg\}
&\Longrightarrow b_Q(pN)>b_Q(p) = p-1\\
\lambda_1(\phi, Q)= \lambda_1(h^j, Q) \text{ for some $j$}\
&\Longrightarrow 
\begin{cases}
b_Q(pN) > b_Q(p)>p-1 & \text{if }\ch K \mid  N  \\
 b_Q(pN) = b_Q(p)>p-1 & \text{otherwise}.
 \end{cases}
\end{align}
\end{lemma}

\begin{proof}
If $\phi(Q) = Q$ and $h(Q) = Q$, then $z=Q$ is a root of $\phi(z) - h^j(z)$ for every $1\leq j \leq p-1$.  This gives the lower bound on $b_Q(p)$.

From Lemma~\ref{derivrel1}, $\lambda_1(h^j, Q)$ is a primitive $p\tha$ root of unity for each $j$.  But also 
$$
\lambda_1(h^j,Q) = \left(\lambda_1(h,Q)\right)^j
$$
 by the definition of the $\lambda_i$ and the fact that $Q \in \Fix(h)$.  Hence, if $i \neq j$, then 
 $$
 \lambda_1(h^i,Q) \neq \lambda_1(h^j,Q).
 $$  So if $\lambda_1(\phi,Q) = \lambda_1(h^j,Q)$, then $\lambda_1(\phi,Q) \neq \lambda_1(h^i,Q)$ for all $1 \leq i \leq p-1$ with $i \neq j$.  

That is, if $b_Q(p)>p-1$, the excess is accounted for by a single factor of $\Psi_{p}(x,y)$.  The rest of the proof proceeds exactly as in Lemma~\ref{an}, focusing on that one factor.
\end{proof}

We conclude the section by stating Proposition~3.2 from~\cite{dynunit}, which will be used in the sequel.  

\begin{prop}\label{fromdynunit}
Let $K$ be a field, $X/K$ a smooth projective curve, and let $\phi\colon X \to X$ be a non-constant morphism defined over $K$ such that $\phi^n$ is non-degenerate \textup(that is, such that the graph of $\phi^N$ and the diagonal intersect properly\textup).  Fix a point $Q \in X$ and define integers $m$, $q$, $r$ by 
\begin{eqnarray*}
m &=& \text{ the primitive period of $Q$ \textup(set $m = \infty$ if $Q \notin \Per(\phi)$\textup),}\\
q &=& \text{the characteristic of $K$,}\\
r &=& \text{the multiplicative period of $\left(\phi^m\right)'(Q)$ in $\overline K^*$}\\
&&\text{\textup(set $r = \infty$ if $m = \infty $ or if $\left(\phi^m\right)'(Q)$ is not a root of unity\textup).}
\end{eqnarray*}
\begin{enumerate}[\textup(a\textup)]
\item
$a_Q^*(N) \geq 0$ for all $N \geq 1$.
\item
Let $N\geq 1$.  Then $a_Q^*(N) \geq 1$ if and only if one of the following three conditions is true:
\begin{enumerate}[\textup(i\textup)]
\item
$N=m$.
\item
$N=mr$. 
\item
$N=q^smr$ for some $s\geq 0$.
\end{enumerate}
\end{enumerate}

\end{prop}

\section{Reducibility results}\label{red}

We now prove one of the main results of this paper.  

\begin{theorem}\label{mainthmred}
Let $\phi$ be a rational map with an automorphism $h$ of prime order~$p$.
Let  $K$ be a field over which $\phi$, $h$, and the points in $\Fix(h)$ are all defined.
\begin{enumerate}[\textup(a\textup)]
\item
If $K$ has characteristic~$0$, then the dynatomic polynomial $\Phi_{pN}^*$ is reducible over $K$ for all but finitely many values of~$N$.
\item
For $K$ of arbitrary characteristic, the dynatomic polynomial $\Phi_{pN}^*$ is reducible over $K$ for all but finitely many prime integers~$N$.
\end{enumerate}
\end{theorem}

The proof of Theorem~\ref{mainthmred} is split into several Propositions.  We outline the argument here:
\begin{itemize}
\item
If $N\geq 1$, then $\widetilde{\Psi_{pN}^*}$ is a polynomial.  (Proposition~\ref{polyprop}.)
\item
For all $N \geq 1$, we have $\widetilde{\Psi_{pN}^*}\mid \Phi_{pN}^*$. (Proposition~\ref{psipnpoly}.)
\item
For all $N> 1$, the $h$-tuned dynatomic polynomials satisfy $\deg \widetilde{\Psi_{pN}^*} < \deg\Phi_{pN}^*$.  In fact, this holds for $N \geq 1$ if $d>2$.  (Proposition~\ref{degarg}.)
\item
The $h$-tuned dynatomic polynomials  are nontrivial for almost all $N$, as described in the statement of the theorem.  (Corollary~\ref{poscharprop} for positive characteristic, and Proposition~\ref{0charprop} for characteristic~$0$.)
\end{itemize}

\begin{prop}\label{polyprop}
With the hypotheses in Theorem~\ref{mainthmred},
$\widetilde{\Psi_{pN}^*}\in K[x,y]$ for all $N\geq 1$.  More specifically, fix a point $Q \in \PP^1$ and define integers $m$, $q$, and $r$ by
\begin{align*}
m &= \text{the primitive $h$-period of $Q$ for $\phi$}\\
& \quad  \text{ \textup(set $m = \infty$ if $Q \notin \Per(\phi)$\textup), }\\
q & = \text{the characteristic of $K$,}\\
r & = \text{the multiplicative period of ${\lambda_1(\phi^m,Q)}/{\lambda_1(h^j,Q)}$, }
\\ & \quad \text{where $1\leq j\leq p-1 $ satisfies $\phi^m(Q) = h^j(Q)$}\\
& \quad \text{\textup(set $r = \infty$ if either $m = \infty $ or if ${\lambda_1(\phi^m,Q)}/{\lambda_1(h^j,Q)}$ }\\
& \quad \text{is not a root of unity for any $1\leq j\leq p-1$\textup).}
\end{align*}

\begin{enumerate}[\textup(a\textup)]
\item
$\widetilde{b_Q^*}(pN) \geq 0$ for all $N\geq 1$.
\item
Let $N\geq 1$.  Then $\widetilde {b_Q^*}(pN) \geq 1$ if and only if one of the following conditions is true.
\begin{enumerate}[\textup(i\textup)]
\item\label{case1}
$N=m$.
\item\label{case2}
$N=mr$,  with $p\nmid r$.  
\item\label{case3}
$N=mrq^s$ for some $s\geq 1$ and $p \nmid r$.
\end{enumerate}
\end{enumerate}
\end{prop}

\begin{proof}
If $\Psi_{pN}^*(x,y)$ is a polynomial, then from the definition $\widetilde{\Psi_{pN}^*}(x,y)$ will be a polynomial as well.  Further, if $Q\in \Fix(h)$, then $\widetilde{b_Q^*}(pN) =0$; otherwise $\widetilde{b_Q^*}(pN) = b_Q^*(pN)$.  Therefore, we prove that if $Q \notin \Fix(h)$, then $b_Q^*(pN)$ satisfies the statement of the proposition.

 From Lemma~\ref{msmall} we see that if $m \nmid N$, then $b_Q(pk) = 0$ for every $k \mid N$.  So also $b_Q^*(pN) = 0$.  We need only consider the case that $m \mid N$.
Now suppose that $p\nmid N$.  The condition that $pk \nmid N$ adds no information, so we may calculate
\begin{equation}
b_Q^*(pN) = \sum_{\substack{k \mid N\\ pk \nmid N}} \mu(N/k)b_Q(pk) = \sum_{k\mid N}\mu(N/k)b_Q(pk).\label{mk}
\end{equation}

For clarity we write $b_Q(\phi,pN)$ for $b_Q(pN)$ because we will be dealing with multiple rational maps.
  Let $\psi = \phi^m$ --- so $Q$ has primitive $h$-period~$1$ for $\psi$ --- and let $n = N/m$.  By the argument above, the only terms which contribute to the sum in equation \eqref{mk} are the ones where $m \mid k$, so 
\begin{align*}
b_Q^*(\phi, pN) &= \sum_{k\mid N}\mu(N/k)b_Q(pk) =
\sum_{k'\mid n}\mu(mn/mk')b_Q(\phi, pmk')\\
&= \sum_{k'\mid n}\mu(n/k')b_Q(\phi^m, pk')= \sum_{k'\mid n}\mu(n/k')b_Q(\psi, pk') \\
&= b_Q^*(\psi, pn)&& \text{since $p\nmid n$.}
\end{align*}
Since $Q$ has primitive $h$-period~$1$ for $\psi$ and $Q\notin \Fix(h)$, we may apply Lemma~\ref{an}.  For clarity of notation, we rename the automorphism so that $\psi(Q) = h(Q)$.

\smallskip
\noindent {\bf Case Ia \quad} $ \left(\frac{\lambda_1(\psi,Q) }{\lambda_1(h,Q)}\right)^{n}\neq 1$.

\smallskip
\noindent {\bf Case Ib \quad}  $ \lambda_1(\psi,Q) = \lambda_1(h,Q)$ and $q \nmid n$.

In both cases, we have $b_Q(\psi,p) = b_Q(\psi,pk)$ for all $k\mid n$.  So then
\begin{align*}
\sum_{k\mid n}\mu(n/k)b_Q(\psi, pk) 
&=  \sum_{k\mid n}\mu(n/k)b_Q(\psi, p) \\
&=\begin{cases}
b_Q(\psi,p) > 0 & \text{ if $n=1$ so that $N=m$}\\
0 & \text{ if $n>1$ so that $N>m$.}\\
\end{cases}
\end{align*} 
Since $b_Q(\psi,p) = b_Q(\phi^N,p) = b_Q(\phi, pN)$, we conclude that in these cases $b_Q^*(pN)=b_Q(pN) > 0 $ if and only if $m=N$.

\noindent{\bf Case II \quad}$ \lambda_1(\psi,Q) = \lambda_1(h,Q)$ and $q \mid n$.

\smallskip
Write $n=q^sM$ with $q\nmid M$. 
Since $\psi(Q) = h(Q)$, we know from the proof of Lemma~\ref{msmall} that 
$$
\psi^{q^i}(Q) = h^{q^i}(Q) = h^j(Q) 
\text{ for some } 0 \leq j \leq p-1.
$$
Since $p\neq q$ (we know $p\nmid n$ but $q \mid n$) we see that in fact $1 \leq j \leq p-1$.
   Then because $ \lambda_1(\psi,Q) = \lambda_1(h,Q)$, Lemma~\ref{creqns} says that also 
   $$
    \lambda_1\left(\psi^{q^i},Q\right) = \lambda_1\left(h^{q^i},Q\right) = \lambda_1\left(h^j,Q\right).
   $$

  If $k\mid M$, then necessarily $q \nmid k$, so we apply equation~\eqref{eqlambdas} to conclude that
$$
b_Q \left(\psi, p q^ik\right) = b_Q\left(\psi^{q^i}, pk\right) = b_Q\left(\psi^{q^i}, p\right) = b_Q\left(\psi, pq^i\right).
$$
We then compute
\begin{align*}
b_Q^*(\psi, pn) 
&=  \sum_{\substack{k \mid n\\ pk\nmid n}}  \mu(n/k)b_Q(\psi, pk) =  \sum_{k\mid n} \mu(n/k)b_Q(\psi, pk) &&\text{since $p\nmid n$}\displaybreak[0] \\
&=  \sum_{k\mid M}\sum_{i=0}^s \mu(q^sM/q^ik)b_Q(\psi, p q^i k) \displaybreak[0]\\
&=  \left(\sum_{k\mid M}\mu(M/k)\right)\left(\sum_{i=0}^s \mu(q^{s-i})b_Q(\psi, p q^i )\right)\displaybreak[0] \\
&=  \begin{cases}
b_Q(\psi, p q^s) -b_Q(\psi, p q^{s-1}) >0 &\text{if $M=1$}\\
0 &\text{if $M>1$.}
\end{cases}
\end{align*}

The fact that $b_Q^*(\psi, pn)>0$ when $M=1$ follows from applying equation~\eqref{eqlambdas} to the difference $b_Q\left(\psi^{q^{s-1}}, pq\right) - b_Q\left(\psi^{q^{s-1}}, p\right)$.  So we have shown that in this case, $b_Q^*(\phi, pN) \geq  0$, and that $b_Q^*(\phi, pN) > 0 $ if and only if $N = mq^s$.

\noindent{\bf Case III \quad } $ \left(\frac{\lambda_1(\psi,Q) }{\lambda_1(h,Q)}\right)^{n} = 1$.

\smallskip
Since $r$ is the exact order of $\frac{\lambda_1(\psi,Q)}{\lambda_1(h,Q)}$ in $\overline{K}^*$, we have $r \mid n$ (and further $r>1$ or we are in Case~Ib or Case~II).  If $r \nmid k$, then by Lemma~\ref{an}, $b_Q(\psi, pk) = b_Q(\psi, p)$.  So we may split the sum of $b_Q^*(\psi, pn)$ into two parts:
\begin{align*}
\sum_{k\mid n}\mu(n/k)b_Q(\psi, pk) 
&= \left(\sum_{\substack{k\mid n \\ r\nmid k}} + \sum_{\substack{k\mid n\\r\mid k}} \right)
\mu(n/k)b_Q(\psi, pk) \\
&= \left(\sum_{\substack{k\mid n \\ r\nmid k}} \mu(n/k)b_Q(\psi, p) \right) + 
 \left(\sum_{\substack{k\mid n \\ r\mid k}} \mu(n/k)b_Q(\psi, pk) \right)\\
 &= \sum_{k\mid n} \mu(n/k)b_Q(\psi, p)  + \sum_{\substack{k\mid n \\ r\mid k}} \mu(n/k)
 \left(b_Q(\psi, pk)-b_Q(\psi,p)\right).\\
 \intertext{ Since $n>1$ the first sum vanishes, so we have}
 \sum_{k\mid n}\mu(n/k)b_Q(\psi, pk) 
 &= \sum_{\substack{k\mid n \\ r\mid k}} \mu(n/k)
 \left(b_Q(\psi, pk)-b_Q(\psi, p)\right). \\
  \end{align*}

 Now if $k\mid n$ and $r\mid k$, then $k = rk'$ for some $k'$ a divisor of $n/r$.  So we can rewrite the final sum~as
 \begin{align*}
\sum_{\substack{k\mid n \\ r\mid k}} \mu(n/k) \left(b_Q(\psi, pk)-b_Q(\psi, p)\right) 
&= \sum_{k' \mid(n/r)}  \mu\left(\frac{n/r}{k'}\right) \left(b_Q(\psi, prk')-b_Q(\psi, p)\right) \displaybreak[0]\\
&= \sum_{k'\mid (n/r)}  \mu\left(\frac{n/r}{k'}\right) \left(b_Q(\psi^r, pk')-b_Q(\psi, p)\right)  \displaybreak[0]\\
&= b_Q ^*(\psi^r, pn/r) - 
\begin{cases}
b_Q(\psi, p) & \text{if $n=r$}\\
0 & \text{if $n> r$.}\\
\end{cases}
\end{align*}

If $n=r$, then
\begin{align*}
b_Q ^*(\phi, pN)
=
b_Q ^*(\psi, pn)
&=
b_Q ^*(\psi^r, p) -b_Q(\psi, p)\\
&=
b_Q(\psi^r, p) -b_Q(\psi, p)
=
b_Q(\psi, pr) -b_Q(\psi,p)
>0.
\end{align*}
And in this case,
\begin{align}
b_Q ^*(\phi, pN) 
&=
b_Q(\psi, pr) -b_Q(\psi, p)\nonumber \\
&=
b_Q(\phi^N, p) -b_Q(\phi^m, p). \label{rtunitcase}\displaybreak[0]\\
\intertext{ If $n>r$, then }
 b_Q ^*(\phi, pN)
&=
b_Q ^*(\psi, pn)
=
b_Q ^*(\psi^r, pn/r). \label{Nr}\displaybreak[0]\\
\intertext{Since $p \nmid r$, }
\psi^r(Q) &= h^r(Q) = h^j(Q) &&\text{ for some }1\leq j\leq p-1,\nonumber\displaybreak[0]\\
\intertext{ and as before we have}
 \frac{\lambda_1(\psi^r,Q)}{\lambda_1(h^j,Q)}
 &=
 \left(\frac{\lambda_1(\psi,Q)}{\lambda_1(h,Q)} \right)^r = 1.\nonumber
 \end{align}
  If $n\neq 0$ in $K$, we may apply Case Ib to $\psi^r$ and conclude that $b_Q ^*(\psi^r, pn/r) =0$ since $n/r >1$.    If $n = 0 $ in $K$, then we can apply Case II to $b_Q^*(\psi^r, pn/r)$ and again conclude that  
 $b_Q ^*(\psi^r, pn/r) =0$ unless $n/r = q^s$ for some $s \geq 1$.

In Case III, we therefore conclude that $b_Q ^*(pN) > 0$ if and only if $n=r$ or $n=q^sr$, so $N=mr$ or $N =mq^s r$.

We must now consider the case that $p \mid N$, so  $N=p^t N'$ where $t\geq 1$ and $p \nmid N'$.  Then the condition that $k \mid N$ but $pk \nmid N$ means that $k = p^t k'$ for some $k'$ that divides $N'$.  So we have:
\begin{align}
b_Q ^*(\phi, pN) & = \sum_{\substack{k \mid N\\ pk \nmid N}} \mu(N/k) b_Q(\phi, pk)=\sum_{k' \mid N'} \mu\left(p^tN'/p^tk'\right) b_Q\left(\phi,p \left(p^t k'\right)\right)\nonumber\\
&= \sum_{k' \mid N'} \mu(N'/k') b_Q\left(\phi^{p^t}, pk'\right)= b_Q ^*\left(\phi^{p^t},pN'\right) \quad\text{since $p \nmid N'$.}\label{pdivncomp}
\end{align}

Suppose that $Q$ has primitive $h$-period $m'$ for $\phi^{p^t}$ and primitive $h$-period $m$ for $\phi$.  We know that the primitive $h$-period for $\phi$  must divide $p^t m'$.  That means $m = p^{t'} m''$ with $0\leq t' \leq t$ and $m'' \mid m'$.  If $t'<t$, then $\phi^{pm}(Q)=Q \neq h(Q)$. Hence $t'=t$ and also $m'' =m'$ by minimality of $m'$.  So in fact  $m = m'p^t$.

We now apply the results above to the map $\phi^{p^t}$.     We conclude that  $b_Q ^*(\phi, pN) \geq 0$ for all $Q\in \PP^1$ and $b_Q ^*(\phi, pN) > 0$ if and only if one of the following conditions hold.
\begin{enumerate}
\item
 $Q$ has primitive $h$-period $N'$ for $\phi^{p^t}$.   So by the argument above, $Q$ has primitive $h$-period  $N=p^tN'$ for $\phi$.

\item
$Q$ has $h$-period $m'$ for  $\phi^{p^t}$, and  $N'=m'r$.   In this case, $N = p^t N' = p^t m' r$.  By the argument above, $Q$ has primitive $h$-period $m=p^t m'$ for $\phi$.  So we have $N=mr$.  Note that $r \mid N'$ so $p \nmid r$.

\item
$Q$ has $h$-period $m'$ for  $\phi^{p^t}$,  and $N'=m'rq^s$.   In this case, $N = p^t N' q^s = p^t m' r q^s$.  So with $m=p^tm'$ we have $N=mrq^s$, and as above, $p \nmid r$.\qedhere
\end{enumerate}
\end{proof}

Proposition~\ref{polyprop} says that the $h$-tuned dynamotic polynomials are indeed polynomials, justifying the name.  We now show that they divide the associated dynatomic polynomials.

\begin{prop}\label{psipnpoly}
For every $N\geq 1$,
$\widetilde{\Psi_{pN}^*} \mid \Phi_{pN}^*$.
\end{prop}

\begin{proof}
To show that $\widetilde{\Psi_{pN}^*} \mid \Phi_{pN}^*$, we must show that $a_Q^*(pN) \geq \widetilde{b_Q^*}(pN)$ for all $Q\in\PP^1$.   If $Q\in \Fix(h)$, then $\widetilde{b_Q^*}(pN) = 0$ for all $N$, and the result is immediate, so we assume $Q \notin \Fix(h)$, in which case  $\widetilde{b_Q^*}(pN) = b_Q^*(pN)$.    We need only consider the case $b_Q^*(pN)>0$; Proposition~\ref{polyprop} describes the three ways this can happen.  Throughout, we assume that $b_Q^*(pN)>0$ and that $Q$ has primitive $h$-period $m$ for some $m \mid N$, and we rename the automorphism so that $\phi^m(Q) = h(Q)$.

\noindent{\bf Case I \quad}
 $N=m$.  

Since $\phi^N(Q) = h(Q)$ and $N$ is the smallest such integer,  we may conclude 
 $\phi^{pN}(Q) = Q$, and in fact $pN $ is the
primitive period of $Q$.
Since $b_Q( pk) = 0$ for $k<N$, we have  $b_Q^*(pN) = b_Q(pN)$.  
Similarly, since $pN $ is the primitive period of $Q$, $a_Q(k) = 0$ for $k< pN $; therefore  $a_Q^*(pN)= a_Q(pN)\geq 1$ by Lemma~\ref{fromdynunit}.  If $b_Q(pN)=1$, we are done.

Otherwise, $b_Q(pN)=e>1$ where  $e$ is the smallest positive integer such that $\lambda_e(\phi^N,Q) \neq \lambda_e(h,Q)$.  By Lemma~\ref{creqns}, then, $\lambda_1(\phi^{pN},Q) = 1$ and $\lambda_i(\phi^{pN},Q) = 0$ for $1<i<e$.  This says that $a_Q( pN)  \geq e$, and we are done in this case.

  \noindent{\bf Case II \quad}
$N=mr$ with $r>1$, $p\nmid r$, and $\frac{\lambda_1(\phi^m,Q)}{\lambda_1(h,Q)}$ a primitive $r^{\text{th}}$ root of unity.

As above, since $Q$ has primitive $h$-period $m$ for $\phi$,  we know that $Q$ has primitive period $pm$ for~$\phi$.  Also, since $\frac{\lambda_1(\phi^m,Q)}{\lambda_1(h,Q)}$ is a primitive $r^{\text{th}}$ root of unity and $p\nmid r$, we have by equation~\eqref{phip}
$$
\lambda_1(\phi^{pm},Q) =  \left(\frac{\lambda_1(\phi^{m},Q)}{\lambda_1(h,Q)}\right)^p
$$
is also a primitive $r^{\text{th}}$ root of unity.

From Proposition~\ref{polyprop} (see equation~\eqref{rtunitcase}), we know that 
$$
b_Q^*(pN) = b_Q(\phi^N, p) - b_Q(\phi^m, p).
$$
  A similar proof in~\cite{dynunit} shows that, since 
$\lambda_1(\phi^{pm},Q) $ is a primitive $r^{\text{th}}$ root of unity,
$$
a_Q^*(pN) = a_Q(\phi^{pN},1) - a_Q(\phi^{pm},1).
$$

Since $\lambda_1(\phi^{pm},Q) \neq 1$ and $\lambda_1(\phi^{m},Q) \neq \lambda_1(h,Q)$, we conclude that $b_Q(\phi^m, p) = a_Q(\phi^{pm},1) = 1$.   It remains to show that if $b_Q(\phi^N, p) =e > 1$, then $a_Q(\phi^{pN},1)  \geq e $, but this follows immediately from  Lemma~\ref{creqns}, exactly as in Case I. Summarizing, we have
 \begin{align*}
a_Q^*(pN) 
&=a_Q(pN) - a_Q(pm) = a_Q(pN) - 1\\
&\geq b_Q(pN) - 1 = b_Q(pN) - b_Q(pm) =b_Q^*(pN).
\end{align*}

  \noindent{\bf Case III \quad}
$N=mrq^s$ with  $\frac{\lambda_1(\phi^m,Q)}{\lambda_1(h^j,Q)}$ a primitve $r^{\text{th}}$ root of unity, and $ K $ has characteristic $q$.  

Once again, the primitive period of $Q$ is $pm$. 
By equation~\eqref{Nr}, we see that 
\begin{align}
b_Q^*(pN) &= b_Q^*(\phi^{mr},p q^s) =b_Q(\phi^{mr},p q^s) -b_Q(\phi^{mr},p q^{s-1})\nonumber \\
&=b_Q(pmrq^s) - b_Q(pmrq^{s-1}),\label{bqcalc}\\
\intertext{and in~\cite{dynunit} a similar argument shows that}
a_Q^*(pN) &= a_Q(pmrq^s) - a_Q(pmrq^{s-1})\label{aqcalc}.
\end{align}

From Proposition~\ref{polyprop} and Lemma~\ref{an}, we conclude that $b_Q(\phi^{mr}, p q^{s-1})=e>1$ and $b_Q(\phi^{mr}, p q^{s}) = f>e>1$.  For ease of notation, we let $\psi = \phi^{mrq^{s-1}}$.  Then we have $\psi(Q) = h^j(Q) $ for some $j$, and furthermore since $e>1$, 
\begin{align}
\lambda_i(\psi,Q) &= \lambda_i(h^j,Q) && \text{ for $1 \leq i < e$, and}
\label{iseq}\\
\lambda_e(\psi,Q) & \neq \lambda_e(h^j,Q). \label{esneq}
\end{align}
Note that $\lambda_e\left((h^j)^p,Q\right) =0$ since $e >1$, so by Lemma~\ref{phij-hj} we have
\begin{align*}
\lambda_e(\phi^{pmrq^{s-1}},Q) &= \lambda_e(\psi^p,Q)
= \lambda_e(\psi^{p},Q)-\lambda_e\left((h^j)^p, Q\right)\\
&= p\left(\prod_{i=1}^{p-1}\lambda_1(\psi,Q_i) \right)(\lambda_e(\psi,Q)-\lambda_e(h^j,Q))\\
&= p \left(\frac{\lambda_1(\psi,Q)}{\lambda_1(h^j,Q)}\right)^p(\lambda_e(\psi,Q)-\lambda_e(h^j,Q))
&&\text{by Lemma~\ref{derivrel2}}\\
&= p (\lambda_e(\psi,Q)-\lambda_e(h^j,Q)) &&\text{by equation~\eqref{iseq}.}\\
\end{align*}
 This can never vanish by equation~\eqref{esneq} and the fact that  $q\neq p$.  Therefore 
\[
a_Q(pmrq^{s-1}) = b_Q(pmrq^{s-1}) = e.
\]
The exact same argument applied to $\psi = \phi^{pmrq^s}$ shows that 
\[
a_Q(pmrq^{s}) = b_Q(pmrq^{s}) = f.
\]
Equations~\eqref{bqcalc} and~\eqref{aqcalc}, then, show that 
$a_Q^*(pN) = b_Q^*(pN)$ in this case. 
\end{proof}

In order to prove that the dynatomic polynomials $\Phi_{pN}^*$ are reducible for infinitely many~$N$, we must be sure that the factors $\widetilde{\Psi_{pN}^*}$ are nontrivial.  The example at the end of Section~\ref{autmapsprelim} shows that this is  not always the case. First, we show that $\deg \widetilde{\Psi_{pN}^*} < \deg \Phi_{pN}^*$ for suitable choice of $N$.   We require one additional Lemma.

\begin{lemma}\label{mucalc}
For $n>1$,  $a\geq 2$ and $p\geq 2$, 
$$\sum_{k\mid n}\mu(n/k)a^{pk} > p \sum_{k\mid n} \mu(n/k)a^k$$
\end{lemma}

\begin{proof}
Let $f(n) = a^{pn}$ and $g(n) = p a^n$, and define $h = (f-g)*\mu$, where $*$ represents convolution in the usual number-theoretic sense.  The statement of the lemma is equivalent to $h(n) > 0$ for all $n>1$.  By properties of convolution (see~\cite{sjmbook} for example), $h*1 = ((f-g)*\mu)*1 = f-g$.  In other words, 
\[
\sum_{k\mid n} h(k) = a^{pn} - p a^n.
\]

We will show by induction that  $h(n) < a^{pn}$ and $h(n) > 0$ for all $n>1$.  Since $a\geq 2$ and $p\geq 2$,
\[
h(1) = a^p - pa < a^p \quad\text{and}\quad
 h(1)= a(a^{p-1}-p)\geq 0.
 \]
In fact, $h(1) = 0 $ if and only if $a=p=2$.
Now consider a prime $q$,
\begin{align*}
h(q) &= a^{pq} - pa^q - h(1) < a^{pq} && \text{since $pa^q > 0 $ and $h(1) \geq 0$.}
\end{align*}
But also
\begin{align*}
h(q) 
&= a^{pq} -pa^q -a^p +pa = a^p(a^{p(q-1)} -1) -pa(a^{q-1}-1)\\
& \geq a^p(a^{q-1}-1)(a^{q-1}+1) - pa(a^{q-1}-1) \quad \text{since $p\geq 2$}\\
& = (a^{q-1}-1)(a^{p+q-1}+a^p-pa) >0 \qquad\quad\text{since $a^p\geq pa$ and $q\geq 2$.}
\end{align*}

Now suppose that for all $k<m$,  we have $h(k) < a^{pk}$ and also that $h(k) > 0$ if $k>1$.  If $m$ is prime, the result holds by the argument above.  Assume then that $m$ is composite (so clearly $m>3$).  For all $k\mid m$, if $k\neq m$, we have $h(k) \geq 0$ by the induction hypothesis, so
$$
h(m) = a^{pm} - pa^m - \sum_{\substack{k\mid m\\k\neq m}}h(k) < a^{pm}.
$$

Also by the induction hypothesis, $h(k) < a^{pk}$ for each $k$ in the sum above.  Further, the largest divisor of $m$ is at most $m-2$ since $m\geq 4$.  Thus, 
$$
\sum_{\substack{k\mid m\\k\neq m}}h(k) 
< a^{p(m-2)} + \cdots +a^p
=
\frac{a^{p(m-1)}-a^p}{a^p-1}.
$$
Using this rough estimate, we find
\begin{align*}
h(m) 
&> a^{pm} -pa^m - \frac{a^{p(m-1)}-a^p}{a^p-1}\\
&= \frac{a^{p(m+1)}-a^{pm}-a^{p(m-1)} +a^p}{a^p-1}-pa^m\\
&= \frac{a^{p(m-1)}\left(a^{2p}-1\right)-a^{pm} + a^p }{a^p-1}-pa^m\\
&= a^{p(m-1)}\left(a^p+1\right) - \frac{a^{pm} - a^p }{a^p-1} - pa^m\\
&> a^{pm} + a^{p(m-1)} - a^{pm} + a^p - pa^m && \text{since the denominator is $> 1$}\\
&= a^{p(m-1)}+a^p -pa^m > 0. &&\qedhere
\end{align*}
\end{proof}

\begin{prop}\label{degarg}
Assume the hypotheses in Theorem~\ref{mainthmred}, and let $\deg\phi = d$.
\begin{enumerate}[\textup(a\textup)]
\item
If $d >2$, then
$\deg\widetilde{\Psi_{pN}^*} < \deg \Phi_{pN}^* $ for all $N\geq 1$. 
\item
If $d =2$, then
$\deg \widetilde{\Psi_{pN}^*}  < \deg \Phi_{pN}^* $ for all $N> 1$. 
\end{enumerate}
\end{prop}

\begin{proof}
From definition~\ref{psidef}, we see that $\deg \widetilde{\Psi_{pN}^*}  \leq \deg \Psi_{pN}^*  $ for every $N$.  So we prove now that
$\deg \Psi_{pN}^*  < \deg\Phi_{pN}^* $ for $N>1$.
Consider first the case that $p\nmid N$, so if $pk\nmid N$ for all $k$. Then
\begin{align}
\deg \Psi_{pN}^* 
&= \sum_{\substack{k \mid N\\ pk \nmid N}} \mu(N/k)(p-1)\left(d^k+1\right)
= \sum_{k \mid N} \mu(N/k)(p-1)\left(d^k+1\right)\nonumber\\
&=
\begin{cases}
(p-1)\left(d+1\right) & \text{ if $N=1$}\\
(p-1)\sum_{k \mid n} \mu(N/k) d^k& \text{ if $N>1$}. \label{degpsin}
\end{cases} 
\end{align}

  Note that if $k \mid pN $ then either $k \mid N$ or $k = pk'$ for some $k' \mid N$.  So
\begin{align*} 
\deg \Phi_{pN}^* &=
 \sum_{k \mid pN} \mu(pN/k)\left(d^k + 1\right)\\
&=
 \sum_{k \mid N} \mu(pN/k)d^k + \sum_{k \mid N} \mu(pN/pk)d^{pk} +\sum_{k \mid pN} \mu(pN/k).
\end{align*}

Since by hypothesis $p\nmid N$, we know that $\gcd(p,k)=1$ for $k$ any divisor of~$N$.   Therefore $\mu(pN/k) = \mu(p)\mu(N/k) = -\mu(N/k)$.  Also, $pN >1$ so the final term vanishes.
\begin{equation}\label{phi2n}
\deg \Phi_{pN}^* =  -\sum_{k \mid N} \mu(N/k)d^k + \sum_{k \mid N} \mu(N/k)d^{pk}
\end{equation}

Comparing equations~\eqref{degpsin} and~\eqref{phi2n}, we see that  $\deg\Psi_{pN}^* < \deg \Phi_{pN}^*$ when $N>1$ follows from the fact that $\sum_{k \mid N} \mu(N/k)d^{pk} > p\sum_{k \mid N} \mu(N/k)d^k $ for all $N>1$, from Lemma~\ref{mucalc}.

In the $N=1$ case, we wish to show that $d^p + 1 - p(d + 1) >0 $ when $d >2$.
 From equations~\eqref{degpsin} and~\eqref{phi2n}, we have
\[
\deg\Phi_p^* -\deg \Psi_p^*= 
\left(d^{p} - d \right) - (p-1)(d+1)
= d^p - pd - (p-1),
\]
which is increasing with $d$, so it will be positive for all $d >2$ if it is positive when $d=3$.
In that case, we calculate
\[
\deg\Phi_p^* -\deg \Psi_p^*
= 3^p - 3p - (p-1),
\]
which is increasing with $p$.  We check that for $p=2$ we have $9- 6 -1=2>0$.

Now if $p\mid N$, then we have $N=p^t N'$ where $t\geq 1$ and $p \nmid N'$.  The condition that $k \mid N$ but $pk \nmid N$ means that $k = p^t k'$ for some $k' \mid N'$.  So we have:
\begin{align}
\deg \Psi_{pN}^* & = \sum_{\substack{k \mid N\\ pk \nmid N}} \mu(N/k)(p-1)\left(d^k+1\right)\nonumber\\
&=(p-1)\sum_{k' \mid N'} \mu\left(p^t N'/p^tk'\right)\left(d^{p^tk'}+1\right)\nonumber\\
&=(p-1)\sum_{k' \mid N'} \mu\left(N'/k\right)\left(d^{p^t}\right)^{k'}+ \sum_{k' \mid N'} \mu(N'/k')
\nonumber\\
&=\begin{cases}
(p-1)\ \sum_{k' \mid N'} \mu\left(N'/k'\right)\left(d^{p^t}\right)^{k'} & \text{if $N'>1$}\\
(p-1)\left(d^{p^t} + 1\right) & \text{if $N'=1$}.  \label{psineven}
\end{cases}
\end{align}

Since $N=p^t N'$,  $pN = p^{t+1}N'$.  If $k \mid pn$ but  $p^t \nmid k$, then $p^2 \mid \left(p^{t+1}N'/k\right)$, which means that $\mu (pN/k) = 0$.  So the only divisors which contribute to the sum below are ones of the form $p^tk'$ where $k' \mid pN'$.
\begin{align}
\deg \Phi_{pN}^* &= \sum_{k \mid pN} \mu(pN/k)\left(d^k + 1\right) \nonumber \\
&= \sum_{k' \mid pN'} \mu(p^{t+1}N'/p^t k')\left(d^{p^tk'} \right) \nonumber\\
&= \sum_{k' \mid pN'} \mu(pn'/ k')\left(d^{p^t} \right) ^{k'}. \label{phineven}
\end{align}
Comparing equations ~\eqref{psineven} and~\eqref{phineven}, we see that if $N'>1$ we are reduced to the case $p\nmid N$ above.  If $N'=1$, the sum in equation~\eqref{phineven} is 
\begin{align*}
d^{p^{t+1}}- d^{p^t} &\geq
4d^{p^t} - d^{p^t} & \text{since $d\geq 2$, $p\geq 2$ and $t\geq 1$}\\
&= 3d^{p^t}\\
&> d^{p^t} + 1 &  \text{since $d\geq 2$, $p\geq 2$ and $t\geq 1$}.
\end{align*}
So again $\deg\Phi_{pN}^* > \deg\Psi_{pN}^*$.
\end{proof}

We must also prove non-triviality in the sense that $\Psi_{pN}^* \neq 1$.  If all roots of $\Psi_{pN}^*$ are in $\Fix(h)$, then the polynomial $\widetilde{\Psi_{pN}^*}$ will be trivial.  So we must first examine the possible values of $b_Q^*(pN)$ when $Q \in \Fix(h)$.

\begin{prop}\label{bq*prop}
With the hypotheses in Theorem~\ref{mainthmred}, let $Q\in \Fix(h)$, and define integers $q$ and $r$ as in Proposition~\ref{polyprop}.  Then
$b_Q^*(pN) \geq 1$ if and only if one of the following conditions is true.
\begin{enumerate}[\textup(i\textup)]
\item\label{case1a}
$N=p^t$ for some $t\geq 0$.
\item\label{case2a}
$N=rp^t$ for some $t\geq 0$.
\item\label{case3a}
$N=rq^sp^t$ for some $t\geq 0$, $s\geq 1$.
\item\label{case4a}
$N=2p^t$ for some $t\geq 0$.
\item\label{case5a}
$N=2rp^t$ for some $t\geq 0$.
\item\label{case6a}
$N=2rq^sp^t$ for some $t\geq 0$, $s\geq 1$.
\end{enumerate}
\end{prop}

\begin{proof}
Let $m$ be the primitive period of $Q$ for $\phi$.  From Lemma~\ref{fplem}, we know that  $m=1$ or~$2$.  So there are really three cases.  
\begin{enumerate}[\textup(i\textup)]
\item
$N=mp^t$ for some $t\geq 0$.
\item
$N=mrp^t$ for some $t\geq 0$.
\item
$N=mrq^sp^t$ for some $t\geq 0$, $s\geq 1$.
\end{enumerate}
In the case $p\nmid N$, the proofs  follow exactly as in Proposition~\ref{polyprop}.  However, we must reconsider the case where $p \mid N$, since we used in an essential way the assumption that $Q \notin \Fix(h)$.

Suppose, then, that $p\mid N$, so write $N = p^t N'$ where $p \nmid N'$.  As in equation~\eqref{pdivncomp}, we find that
$$
b_Q^*(\phi, pN) = b_Q^*\left(\phi^{p^t}, pN'\right).
$$
Let $m'$ be the primitive period of $Q$ for $\phi^{p^t}$.  We must have $m' = 1$ if $m=1$ or if $m=p=2$, and $m'=2$ otherwise.
Applying the three cases where $p \nmid N$ to $b_Q^*\left(\phi^{p^t}, pN'\right)$, we see that either 
$N' = m'$, or $N' = m'r$, or $N'=m'q^s$.  
Substituting each of these for $N'$ gives the desired result.
\end{proof}

We can now prove a nontriviality result for general fields $K$ by showing that points of primitive $h$-period $\ell$ for $\phi$ exist for almost all primes $\ell$.  This proof is essentially the same as the proof of existence of points of primitive period $\ell$ found, for example, in~\cite[page~$154$]{ads}.  The following weak result holds in general.  In Proposition~\ref{0charprop}, we prove a much stronger result for rational maps defined over a field of characteristic~$0$. 

\begin{prop}
Under the hypotheses of Theorem~\ref{mainthmred} with $\deg\phi = d$, for all prime numbers $\ell$ except for at most $d+6$ exceptions, the map $\phi$ has a point of primitive $h$-period $\ell$.
\end{prop}
 
 \begin{proof}
 We begin by discarding the finitely many primes satisfying any of the following conditions:
 \begin{itemize}
 \item
 $\ell = 2$.
 \item
 $\ell = p$.
 \item
 $\ell = \ch K $.
 \item
 There is some $Q$ with primitive $h$-period~$1$ and some $1 \leq j \leq p-1$ such that
 $\frac{\lambda_1(\phi,Q)}{\lambda_1\left(h^j, Q\right)}$ is a primitive $\ell\tha $ root of unity.
  \item
 There is some $Q\in \Fix(h)$ and some $1 \leq j \leq p-1$ such that
 $\frac{\lambda_1(\phi,Q)}{\lambda_1\left(h^j, Q\right)}$ is a primitive $\ell\tha $ root of unity.
 \end{itemize}

There are at most $d+1$ points of primitive $h$-period~$1$, and the set $\Fix(h)$ has at most two elements, so this list eliminates at most $d+6$ primes.  Note that we have eliminated the primes where $b_Q^*(p\ell)\geq 1$ for $Q\in \Fix(h)$.

For any of the remaining primes $\ell$, consider a root $Q$ of $\Psi_{p\ell}^*(x,y)$.  Then $Q$ must be a point of $h$-period $\ell$, and furthermore $Q \notin \Fix(h)$. 
 If $Q$ does not have primitive $h$-period $\ell$, it must have primitive $h$-period~$1$ by Lemma~\ref{msmall}.  Because of the primes we have eliminated, and by results in Proposition~\ref{an}, we see that 
 $b_Q(p\ell) = b_Q(p)$, so 
\begin{align*}
 \sum_{\text{roots of }\Psi_p\cap \text{roots of }\Psi_{p\ell}} b_Q(p\ell)
 & =  \sum_{\text{roots of }\Psi_p \cap \text{roots of }\Psi_{p\ell} } b_Q(p)\\
 & \leq  \sum_{\text{roots of }\Psi_p} b_Q(p) = d+1.
\end{align*}
 That is, the total multiplicity of all roots of $\Psi_{p\ell}$ that do not have primitive $h$-period $\ell$ is at most $d+1$.  But the degree of $\Psi_{p \ell}$ is $d^{\ell}+1$.  So $\Psi_{p \ell}$  has at least one root that is a point of primitive $h$-period $\ell$.
   \end{proof}

\begin{cor}\label{poscharprop}
For a rational map $\phi$ with $\deg\phi=d$, under the hypotheses of Theorem~\ref{mainthmred}, there exist at most $d+6$ primes $\ell$ such that the polynomial $\widetilde{\Psi_{p\ell}^*}$ is trivial.
\end{cor}
\begin{proof}
This follows immediately from the result above.
\end{proof}

As in the case of periodic points, a stronger result is possible if we restrict ourselves to characteristic~$0$.  The following result parallels one by I.N.~Baker for periodic points in~\cite{inbaker}.

\begin{prop}\label{0charprop}
Let $\phi$ be a rational map of degree $d\geq 2$ defined over a field~$K$ of characteristic~$0$, and let $h$ be an automorphism for $\phi$ of prime order~$p$.  Suppose that $\phi$ has no points of primitive $h$-period $N>1$ in $\overline K$.   Then 
\[
(d,N) \in \left\{(2, 2), (2, 3), (2, 4),  (3, 2),  (4, 2)\right\}.
\]
\end{prop}

\begin{proof}
 Suppose that $\phi$ is as described, and that $\phi$ has no points of primitive $h$-period~$N$.  Then all roots of the (nontrivial) polynomial $\Psi_{pN}(x,y)$ are accounted for by points of primitive $h$-period $m<N$ or by points in $\Fix(h)$.  Let
\[
 S =\left\{ Q \in \PP^1 : \Psi_{pN}(Q) = 0\right\},
\]
 and for each $Q \in S$, let 
\[
 m_Q = 
\begin{cases}
\text{the primitive $h$-period of $Q$ for $\phi$} &\text{if $Q \notin \Fix(h)$}\\
\text{the primitive period of $Q$ for $\phi$ (necessarily $1$ or $2$)} &\text{if $Q \in \Fix(h)$}.
\end{cases}
\]
 By Lemma~\ref{msmall} and established properties of periodic points (see~\cite{beardon}, for example), we know that each $m_Q \mid N $.  So let
\[
 M = \left\{ m \in \Z : 1 \leq m < N \text{ and } m \mid N \right\}.
\]

We now compute lower and upper bounds for
\begin{equation}\label{sumineq}
\sum_{Q \in S} \left(b_Q(pN) - b_Q\left(pm_Q\right) \right),
\end{equation}
under the assumption that $Q\in S$ implies that $m_Q\in M$; that is, under the assumption that there are no points of primitive $h$-period $N$.
For the lower bound,
\begin{align}
\sum_{Q \in S}b_Q(pN) &= \deg \Psi_{pN} = (p-1)\left(d^n +1\right)\label{est1}\\
\sum_{Q \in S}b_Q(pm_Q )
&= \sum_{m\in M}\sum_{\substack{m_Q= m\\ Q \in S}} b_Q(p m_Q )\nonumber\\
&\leq \sum_{m\in M} \deg \Psi_{pm} = \sum_{m\in M} (p-1)\left(d^m + 1\right).\label{est2}
\end{align}
Now, when $N=2$, the set $M=\{1\}$, so the final sum in equation~\eqref{est2} is exactly 
\[
(p-1)(d+1) =(p-1)\left( d^{N-1} + (N-1) \right).
\]
  If $N>2$, then $\gcd(N, N-1) = 1$ and so 
\begin{equation*}
\sum_{m\in M} \left(d^m + 1\right) \leq \sum_{i=1}^{N-2} \left(d^i + 1\right)  \leq d^{N-1} + N-1.
\end{equation*}
So we have our lower bound:
\begin{equation}\label{lower}
(p-1)\left(d^N - d^{N-1} - (N - 2) \right) \leq \sum_{Q \in S} \left(b_Q(pN) - b_Q\left(p m_Q\right) \right) .
\end{equation}

To compute the upper bound, we will use the assumption that all of the points in $S$ are roots of $\Psi_{pm}$ for some $m\in M$.  Then from Lemmas~\ref{an} and~\ref{bnqfix} applied to the map $\phi^m$, we see that $b_Q(pN)-b_Q(pm) > 0 $ if and only if 
\[
\frac{\lambda_1\left(\phi^m,Q\right)}{\lambda_1\left(h^j,Q\right)}
\]
is a primitive $r\tha$ root of unity for some $r $ not divisible by $p$ and some $1 \leq j \leq p-1$.  Equation~\eqref{phip} (again applied to $\phi^m$) shows that $\lambda_1(\phi^{pm},Q)$ must then be a primitive $r\tha $ root of unity.  In other words, $Q$ must be on a rationally indifferent cycle of length $pm$.

Also, by Proposition~\ref{psipnpoly}, we know that $a_Q(pN) \geq b_Q(pN)$ for every $N$.  So we may now compute the upper bound:
\begin{align*}
 \sum_{Q \in S} \left(b_Q(pN) - b_Q\left(p m_Q\right) \right) 
& =
 \sum_{\substack{Q \in S \\ Q \text{ on a rationally}\\ \text{ indifferent cycle}} }
 \left(b_Q(pN)- b_Q\left(p m_Q\right) \right)  \displaybreak[0]\\
 & \leq 
  \sum_{\substack{Q \in S \\ Q \text{ on a rationally}\\ \text{ indifferent cycle}} }b_Q(pN)\displaybreak[0]\\\
&  \leq
   \sum_{\substack{Q \in S \\ Q \text{ on a rationally}\\ \text{ indifferent cycle}} }a_Q(pN).
\end{align*}

Beardon provides exactly the upper bound we require; in~\cite[page~$146$]{beardon} he shows that 
\begin{equation}\label{upper}
   \sum_{\substack{Q \in S \\ Q \text{ on a rationally}\\ \text{ indifferent cycle}} }a_Q(pN)
   \leq pN(d-1).
\end{equation}

By hypothesis, $N>1$.  We now show that if $d>4$,  the inequality 
\begin{equation}
(p-1)\left(d^N - d^{N-1} - (N - 2) \right)
\leq
pN(d-1)
\end{equation}
cannot hold.
Since $\frac {p-1}{p}\geq \frac 1 2$, we will instead use the inequality
\begin{equation}\label{ineq}
\frac 1 2\left(d^N - d^{N-1} - (n - 2) \right)
\leq
N(d-1).
\end{equation}
 Also,
$\frac 12\left(d^N - d^{N-1} - (n - 2) \right) \geq \frac 12\left(d^N - d^{N-1}\right) -N$, so we have
\begin{align}
\frac 1 2 \left(d^N - d^{N-1} \right)
&\leq
Nd\nonumber \\
\frac 1 2 \left(d^{N-1} - d^{N-2} \right)
&\leq
N \nonumber\\
\frac{d-1}{2}d^{N-2}
&
\leq N\label{d-1/2}\\
2d^{N-2} &\leq N, \nonumber
\end{align}
where the last step follows from the assumption that $d>4$.
The function $2d^{N-2} - N$ is increasing with $N$ and is~$0$ when $N=2$.   Going back to the original inequality~\eqref{ineq}, we check that for $N=2$ and $d>4$ it still cannot hold:
\begin{align*}
\frac 1 2 \left(d^2 - d   \right)
&\leq
2(d-1)\\
d(d-1) & \leq 4(d-1)\\
d& \leq 4,
\end{align*}
which contradicts $d > 4$.

When $d=4$, the inequality in equation~\eqref{d-1/2} becomes
$\frac{3}{2}4^{N-2}
\leq N$.  Again, we see that the function $\frac{3}{2}4^{N-2} - N$ is increasing with $N$ and it is already positive when $N=3$, so the inequality can hold only when $N=2$.  

Similar computations show that when $d=3$ and $N\geq 3$, or when $d=2$ and $N\geq 5$, then inequality~\eqref{ineq} cannot hold.
\end{proof}

With this, we have completed the proof of Theorem~\ref{mainthmred}.   A consequence of this theorem is the following.

\begin{cor}\label{mainthmcor}
The moduli space $M_d(pN, \mathfrak C_p)$ is geometrically reducible for 
\begin{enumerate}[\textup(a\textup)]
\item
all but finitely many integers $N$ if $K$ has characteristic~$0$, and
\item
all but finitely many prime integers $N$ for arbitrary fields $K$.
\end{enumerate}
\end{cor}

\begin{proof}
This result for the space $\Rat_d(pN, \mathfrak C_p)$ follows immediately from the work above.  We may fix the automorphism $h(z) = \zeta_p z$, and write an arbitrary map $\phi(x,y)$ as in~\eqref{phiform}.   Then whenever $ \widetilde{\Psi_{pN, \phi, h}} $ is nontrivial, its vanishing defines a proper closed subvariety $X \subset \Rat_d(pN, \mathfrak C_p)$.

The image of $X$ under the projection $\Rat_d(pN, \mathfrak C_p) \to M_d(pN, \mathfrak C_p)$ is clearly closed.  
We know that 
\begin{align*}
X = \{ (\alpha, \phi) \in \PP^1 \times \Rat_d \colon
&\alpha \text{ is a point of formal period  $pN$ for $\phi$ },\\
& \Aut(\phi) \text{ contains the subgroup } \langle \zeta_p z \rangle, \text{ and } \\
&\zeta_p \alpha \in \mathcal O_{\phi}(\alpha) \}.
\end{align*}
On the other hand, if $Y = \Rat_d(pN, \mathfrak C_p) \smallsetminus X$ then 
\begin{align*}
Y =
 \{ (\alpha, \phi) \in \PP^1 \times \Rat_d \colon
&\alpha \text{ is a point of formal period  $pN$ for $\phi$ },\\
& \Aut(\phi) \text{ contains the subgroup } \langle \zeta_p z \rangle, \text{ and } \\
&\zeta_p \alpha \not \in \mathcal O_{\phi}(\alpha) \}.
\end{align*}

The subvarieties $X$ and $Y$ are clearly not equivalent under the $PGL_2$ action, so their images under the projection $\Rat_d(pN) \to M_d(pN)$ will be disjoint.  Hence, we see that the image of $X$ will be a proper closed subscheme of $M_d(pN, \mathfrak C_p)$.
\end{proof}

\section{Reducibility for pure power functions}\label{power}

We are able to provide a complete description of how the dynatomic polynomials factor in the case of the pure power functions $z^d$ and their reciprocals $z^{-d}$.  Note that both maps have the order-$2$ automorphism $z \mapsto 1/z$.  Additionally, $\phi(z) = z^d$ has a $z\mapsto \zeta_{d-1}z$ automorphism and $\phi(z) = \frac 1 {z^d}$ has a $z\mapsto \zeta_{d+1}z$ automorphism, where $\zeta_k$ represents a primitive $k\tha$ root of unity.

\begin{lemma}\label{z2case}
Let $\phi(z) = z^d$ for $d\geq 2$.  Then for $N>1$,
\begin{equation}\label{powerfactor}
\Phi_{n,\phi}^*(z) = \prod_{\substack{k\mid d^N - 1\\  k \nmid d^m - 1, \\
m\mid n, m\neq n}} C_k(z),
\end{equation}
where $C_k(z)$ is the $k\tha$ cyclotomic polynomial in $z$.
\end{lemma}

\begin{remark}
We note that this fact has appeared in the literature, for example in~\cite{algcurve}.  As we could not find a proof, we provide one here for completeness.
\end{remark}

\begin{proof}
First we show that equation~\eqref{powerfactor} holds for $N$ prime.
\begin{align*}
\Phi_{N,\phi}^*(z) 
&= \prod_{k|N}\left(z^{d^k}-z\right)^{\mu(N/k)}
= \frac{z^{d^N-1}-1}{z^{d-1}-1} \\
&= \frac{\prod_{k\mid d^N-1}C_k(z)}{\prod_{k\mid d-1}C_{k}(z)}
=  \prod_{\substack{k\mid d^N - 1\\ k \nmid d^m - 1, \\m\mid n, m\neq n}} C_k(z).
\end{align*}

Now assume that the result holds for all $n< N$.  Since we have the result for primes, we assume $N$ is 
composite, and write $N=p^e n$ for some prime $p$, with $e\geq 1$ and $p \nmid n$.  If $k\mid N$ and $p^{e-1} \nmid k$, then $p^2 \mid (N/k)$, which gives $\mu(N/k)=0$; in other words, $k$ does not contribute to the product $\Phi_N^*$.  Further, if $p^{e-1}k \mid N$, then $k\mid pn$, so either $k = k'$ or $k=pk'$ for some $k'\mid n$, and these sets are disjoint since $p\nmid n$.  Finally, note that since $p\nmid n$ (and hence $p\nmid k$ for any $k\mid n$), we have $\mu(p^en/p^{e-1}k) = \mu(pn/k)=-\mu(n/k)$,  Putting this all together, we find

\begin{align*}
\Phi_{N}^*(z) &= \prod_{k|N}\left(z^{d^k}-z\right)^{\mu(N/k)}\\
&=\left(\prod_{k|n}(z^{d^{(p^{e-1}k)}}-z)^{\mu(p^en/p^{e-1}k)}\right)
\left(\prod_{k|n}(z^{d^{(p^{e}k)}}-z)^{\mu(p^en/p^{e}k)}\right) \displaybreak[0] \\
&=\left(\prod_{k|n}\left(z^{\left(d^{p^{e-1}}\right)^k}-z\right)^{-\mu(n/k)}\right)
\left(\prod_{k|n}\left(z^{\left(d^{p^{e}}\right)^k}-z\right)^{\mu(n/k)}\right) \displaybreak[0] \\
&=\left(\prod_{\substack{k\mid \left(d^{p^{e-1}}\right)^n - 1\\ k \nmid  \left(d^{p^{e-1}}\right)^m - 1, \\ m\mid n, m\neq n}} \frac{1}{C_k(z)}\right)
\left(\prod_{\substack{k\mid \left(d^{p^{e}}\right)^n - 1\\ k \nmid  \left(d^{p^{e}}\right)^m - 1, \\ m\mid n, m\neq n}} C_k(z)\right) =  \prod_{\substack{k\mid d^N - 1\\ k \nmid d^m - 1, \\ m\mid N, m\neq N}}C_k(z).
\end{align*}

  The penultimate equality above follows from the induction hypothesis because $n<N$.  The final equality follows from the fact that if $k\mid d^m-1$ for some $m\mid N$ and $m\neq N$, then $k\mid d^{p^e m'} - 1$ for some $m'\mid n$ or $k\mid d^{p^{e-1}n}-1 $.  (Note the divisors of $d^{p^{e-1}m}-1 $ for some $m\mid n$ are included in the first set.)
\end{proof}

\begin{example}
Let $\phi(z) = z^2$.  Here are factorizations for the first few dynatomic polynomials:
\begin{align*}
\Phi_1^*(z) &= z(z-1) \displaybreak[0]\\
\Phi_2^*(z) &= z^2+z+1 \displaybreak[0]\\
\Phi_3^*(z) &= z^6+z^5+z^4+z^3+z^2+z+1 \displaybreak[0]\\
\Phi_4^*(z) &= \left(z^4+z^3+z^2+z+1\right)
   \left(z^8-z^7+z^5-z^4+z^3-z+1\right)\displaybreak[0]\\
\Phi_5^*(z) &= z^{30}+z^{29}+z^{28}+z^{27}+z^{26}+z^{25}+z^{24}+z^{23}+z^
   {22}+z^{21}+z^{20}\\
   & \qquad +z^{19}+z^{18}+z^{17}+z^{16}+z^{15}+z
   ^{14}+z^{13}+z^{12}+z^{11}+z^{10}\\
   & \qquad \quad+z^8+z^7+z^6+z^5+z
   ^4+z^3+z^2+z+1\displaybreak[0]\\
\Phi_6^*(z) &= \left(z^6+z^3+1\right)
   \left(z^{12}-z^{11}+z^9-z^8+z^6-z^4+z^3-z+1\right)\\
& \qquad   \left(z^{36}-z^{33}+z^{27}-z^{24}+z^{18}-z^{12}+z^9-z^3
   +1\right).
\end{align*}

\end{example}

\begin{lemma}\label{1/z2case}
Let $\phi(z) = 1/z^d$ for $d\geq 2$.  Then for $N>2$,

\begin{eqnarray}
2\nmid N
 &\Longrightarrow& 
\Phi_{N,\phi}^*(z) = 
\prod_{\substack{k\mid d^N + 1\\ k \nmid d^m + 1, \\m\mid N, m\neq N}} C_k(z) \label{oddN}\\
N=2n \text{ with } 2\nmid n
 &\Longrightarrow& 
\Phi_{N,\phi}^*(z) = 
\prod_{\substack{k\mid d^N - 1\\ k \nmid d^{2m} - 1, \\m\mid n, m\neq n \\ k\nmid d^n+1}} C_k(z)\label{evenN} \\
N=2^en \text{ with } e\geq 2 \text{ and } 2\nmid n
 &\Longrightarrow& 
 \Phi_{N,\phi}^*(z) = 
\prod_{\substack{k\mid d^N - 1\\ k \nmid d^m - 1, \\m\mid N, m\neq N}} C_k(z) \label{squareN}
\end{eqnarray}

\end{lemma}

\begin{proof}
If $\phi(z) = \frac{1}{z^d}$, then for $k$ odd we have $\phi^k(z) = \frac{1}{z^{d^k}}$, and for $k$ even we have $\phi^k(z) = z^{d^k}$.  So we may calculate
\begin{equation}\label{twoprods}
\Phi_{N,\phi}^*(z) =
\left(\prod_{\substack{k\mid N \\ 2\mid k}} \left(z^{d^k}-z\right)^{\mu(N/k)}\right)
\left(\prod_{\substack{k\mid N \\ 2\nmid k}} \left(z^{d^k+1}-1\right)^{\mu(N/k)}\right).
\end{equation}

If $4\mid N$, then for any odd divisor $k$ of $N$, we have $\mu(N/k)=0$, so the second product is empty, and the first may be taken over \emph{all} divisors of $N$.  So we rewrite the product as
\begin{align*}
\prod_{k\mid N} \left(z^{d^k}-z\right)^{\mu(N/k)}
&=
\left( \prod_{k\mid N}z^{\mu(N/k)}\right)
\left( \prod_{k\mid N} \left(z^{d^k-1}-1\right)^{\mu(N/k)}\right)\\
&=
\prod_{k\mid N} \left(z^{d^k-1}-1\right)^{\mu(N/k)},
\end{align*}
which simplifies to the expression in equation~\eqref{squareN} by the argument given in Lemma~\ref{z2case}.

If $N$ is odd, it has no even divisors; so the first term is an empty product and the second may be taken over all divisors of $N$.  The argument in this case follows as in Lemma~\ref{z2case}.  Note that since $N$ is odd, $d^m+1 \mid d^N+1$ when $m\mid N$.

It remains only to consider the case that $2\mid N$ but $4 \nmid N$.  So $N=2n$ with $n$ odd.  Then any odd divisor $k$ of $N$ is simply a divisor of $n$.  We see that for such divisors, $\mu(N/k) = \mu(2n/k) = -\mu(n/k)$.  So the second term in equation~\eqref{twoprods} is
\begin{eqnarray}
\prod_{\substack{k\mid N \\ 2\nmid k}} \left(z^{d^k+1}-1\right)^{\mu(N/k)}
&=&
\prod_{k\mid n } \left(z^{d^k+1}-1\right)^{-\mu(n/k)}\nonumber\\
&=&
\prod_{\substack{k\mid d^n+1 \\ k\nmid d^m+1, m\mid n}} \frac{1}{C_k(z)},  \label{odddiv}
\end{eqnarray}
by the argument for the case when $N$ is odd.

Similarly, we recognize that even divisors of $N$ are of the form $2k$ where $k\mid n$, and in this case $\mu(N/2k) = \mu(2n/2k) = \mu(n/k)$.  The first term in equation~\eqref{twoprods} then becomes
\begin{eqnarray}
\prod_{\substack{k\mid N \\ 2\mid k}} \left(z^{d^k-1}-1\right)^{\mu(N/k)}
&=&
\prod_{k\mid n } \left(z^{d^{2k}-1}-1\right)^{\mu(n/k)}\nonumber\\
&=&
\prod_{\substack{k\mid d^N-1 \\ k\nmid d^{2m}-1, m\mid n}} C_k(z).\label{evendiv}
\end{eqnarray}
Multiplying equations~\eqref{evendiv} and~\eqref{odddiv}, we get precisely the expression in equation~\eqref{squareN}.
\end{proof}

\begin{example}
Let $\phi(z) = 1/z^2$.  Here are factorizations for the first few dynatomic polynomials:
\begin{align*}
\Phi_1^*(z) &=(1-z) \left(z^2+z+1\right)\\
\Phi_2^*(z) &= -z\\
\Phi_3^*(z) &= z^6+z^3+1\\
\Phi_4^*(z) &= \left(z^4+z^3+z^2+z+1\right)
   \left(z^8-z^7+z^5-z^4+z^3-z+1\right)\\
\Phi_5^*(z) &= \left(z^{10}+z^9+z^8+z^7+z^6+z^5+z^4+z^3+z^2+z+1\right)\\
  & \qquad  \left(z^{20}-z^{19}+z^{17}-z^{16}+z^{14}-z^{13}+z^{11}-
   z^{10}+z^9 \right.\\
   & \qquad \qquad\left. -z^7+z^6-z^4+z^3-z+1\right)\\
\Phi_6^*(z) &= \left(z^6+z^5+z^4+z^3+z^2+z+1\right)\\
&\qquad   \left(z^{12}-z^{11}+z^9-z^8+z^6-z^4+z^3-z+1\right)\\
& \qquad \quad  \left(z^{36}-z^{33}+z^{27}-z^{24}+z^{18}-z^{12}+z^9-z^3
   +1\right).
\end{align*}

\end{example}

\begin{cor}\label{reducecor}
\qquad
\begin{enumerate}[\textup(a\textup)]
\item\label{zdred}
Let $\phi(z) = z^d$.  If $d>2$, then $\Phi_{N,\phi}^*(z)$ is reducible for every $N$.  If $d=2$, then $\Phi_{N,\phi}^*(z)$ is irreducible if and only if $2^N-1$ is prime.
\item\label{1/zdred}
Let $\phi(z) = 1/z^d$.  If $d>2$, then $\Phi_{N,\phi}^*(z)$ is reducible for every $N$.  If $d=2$, then $\Phi_{N,\phi}^*(z)$ is reducible for every $N\neq 2$ or~$3$.
\end{enumerate}
\end{cor}

\begin{proof}
\eqref{zdred}. $\Phi^*_1(z) = z^d-z$, which is reducible.  Lemma~\ref{z2case} says that $\Phi_{N,\phi}^*(z)$ for $N>1$ is irreducible if and only if $d^N-1$ has no divisors other than those which divide $d^m-1$ where $m\mid N$.
We always have $d-1 \mid d^N-1$, and if $d>2$, so the the quotient $d^{N-1} + \cdots +d+1$ does not divide $d^m-1$ for any $m<N$.  So then $\Phi_{N,\phi}^*(z)$ is reducible for every $N$.

In the case $d=2$, we see that if $2^N-1$ is prime, then $\Phi_{N,\phi}^*(z)=C_{2^N-1}(z)$ is irreducible.  If $N$ is prime but $2^N-1$ is not prime, then $2^N-1$ has a factor not of the form $2^m-1$ with  $m \mid N$.  So then $\Phi_{N,\phi}^*(z)$ is reducible.  Finally, if $N$ is composite, let $k$ be the smallest positive divisor of $N$.  We argue as above that  the quotient of $2^N-1$ and $2^k-1$ does not divide $2^m-1$ for any $m\neq N$.  So then $\Phi_{N,\phi}^*(z)$ has at least two nontrivial factors, namely $C_{2^N-1}(z)$ and $C_{(2^N-1)/(2^k-1)}(z)$.

\eqref{1/zdred}.  In this case, $\Phi^*_1(z) = z^{d+1}-1$, which is reducible since $d\geq 2$.  If $4\mid N$,  then $2^N-1$ is not prime, so the fact that $\Phi_{N,\phi}^*(z)$ is reducible follows from the fact that it is identical to $\Phi_{N,\phi}^*$ for the map $\phi(z) = z^d$, which is reducible in this case by the argument above.  

If $N$ is odd, we need to check that $d^N+1$ has a factor which does not divide $d^m+1$ for any $m\mid N$.  Let $k$ be the smallest nontrivial divisor of $N$.  Then $(d^N+1)/(d^k+1) = d^{N-k}-d^{N-2k}+\cdots+1$.  For this quotient to divide $d^m+1$, it is certainly necessary that $N-k <m$,  But with $N\geq 4$ and $k$ the smallest divisor of $N$, we have $N-k \geq N/k$ which is the largest divisor of $N$.  So we have at least two nontrivial factors of $\Phi_{N,\phi}^*$, namely $C_{d^N+1}(z)$ and $C_{(d^N+1)/(d^k+1)}(z)$ with $k$ the smallest positive divisor of $N$.

Finally, if $2\mid N$ but $4 \nmid N$, we see from equation~\ref{evenN} that for $m$ an odd divisor of $N$, we have $C_{d^m-1} (z) \mid \Phi_{N,\phi}^*(z)$.   In particular, $C_{d-1} (z) \mid \Phi_{N,\phi}^*(z)$ is a nontrivial factor as long as $d>2$.  If $d=2$, we see from the example above that $\Phi_1^*$ is reducible and that $\Phi_2^*$ and $\Phi_3^*$ are not.  If $N>3$, we are assured of an odd factor of $N$ greater than one, and so all such $\Phi_{N,\phi}^*$ will be reducible.
\end{proof}

\section{An irreducibility result}

We now focus on the case $\deg\phi=2$, and prove that if $2^N-1$ is prime, then the dynamic modular curve $M_2(N, \mathfrak C_2)$ is irreducible.

\begin{lemma}\label{multform}
Let $\phi(z)$ be a rational map of degree $d=2$ with a automorphism group of order~$2$ over a field~$K$ with characteristic different from  $ 2$.   Then $\phi$ is $\PGL_2$-conjugate to a unique map of the form
\begin{equation}
\psi(z) = \frac{z^2+az}{az+1},
\end{equation}
where $\psi$ is defined over some finite extension of $K$ and $a^2\neq 1$.
\end{lemma}

\begin{proof}
The proof is identical to the one Milnor gives
in~\cite{milnrat} for maps defined over~$\C$.  We sketch the argument here to show that it works for other fields as well.  In the remarks before Lemma~\ref{leadx}, we showed that $\phi$ must have three distinct fixed points (the only map with a single fixed point has automorphism group isomorphic to $\mathcal S_3$, so we may disregard this case).  We know that $h \in \Aut(\phi)$ for some $h$ of order~$2$.  Since any nontrivial element of $\text{PGL}_2$ cannot fix three points in $\PP^1$, we see that  $h$ must interchange two fixed points of $\phi$. 

 We conjugate by an appropriate $f\in \PGL_2\left(\overline K\right)$ so the two fixed points interchanged by $h$ move to~$0$ and~$\infty$.  Then 
\[
\phi^f(z) = z\left(\frac{az+b}{cz+d}\right).
\]
(Note that the fixed points may be elements of some cubic extension of $K$ --- in fact, we will see in the next section that they are at most in a quadratic extension of $K$ --- so $\phi^f$ is defined over the field we get by adjoining the fixed points of $\phi$ to $K$.)
We have $\deg\phi=2$, so also $\deg\phi^f = 2$.  We conclude that $ad \neq 0$ and $ad-bc \neq 0$.  Dividing the numerator and denominator by $d$ leaves $\phi$ unchanged as a rational map, so we may assume that $d=1$.  Finally, we conjugate by $g(z) = z/a$, which gives the map
\[
\psi(z) = z\left(\frac{z+b}{(c/a) z+1}\right),
\]
 where $b$ and $c/a$ are the multipliers of the fixed points at $0$ and $\infty$.  By construction, the map $\psi$ has  an automorphism that interchanges the fixed points at~$0$ and~$\infty$.  If an automorphism of a rational map interchanges two fixed points, then the fixed points have equal multipliers. We conclude that the multipliers at~$0$ and~$\infty$ are equal.  That is,  $b = c/a$, and $\psi$ has the desired form.   (The fact that $a^2\neq 1$ follows from the fact that $\deg \psi = 2$.)

To see that the value of $a$ is unique, we note that if $\psi^f$ has the same form, then $f$ must fix the set $\left\{0, \infty\right\}$, so $f(z) = \alpha z$ or $f(z) = \alpha/z$ for some $\alpha \in \overline K^*$.  A calculation shows that in either case, $\alpha=1$ is necessary to preserve the form of $\psi$.
\end{proof}

Let $M_2(N,\mathfrak C_2)$ be the moduli space of degree-2 rational maps with an automorphism of order~$2$, together with a point of formal period $N$.  For each $N$, this is an algebraic curve lying in the moduli space $M_2(N)$.
 Lemma~\ref{multform} says that all but one of these maps form a one-parameter family in $M_2$, parameterized by $a$.  We use this fact to prove our irreducibility result. 

\begin{prop}\label{primeirred}
If $2^N-1$ is prime, then the curve $M_2(N,\mathfrak C_2)$ is irreducible.
\end{prop}

\begin{proof}
Consider the dynatomic polynomials $\Phi^*_{N,\phi_a}(x,y) = \Phi^*_N(x,y,a)$ with $\phi_a$ given by the normal form from Lemma~\ref{multform}.  
The curves $M_2(N, \mathfrak C_2)$ can be defined by the vanishing locus $ \Phi^*_N(x,y,a)=0$.  We will prove the curves are irreducible by proving that the dynatomic polynomials are.

Suppose that $\Phi_N^*(x,y) = A(x,y)B(x,y)$ with $\deg_x(A), \deg_x(B) \geq 1$.  Exactly as in Lemma~\ref{irred2}, we see that specializing to $a=0$ will cause neither factor to become trivial.
But specializing to $a=0$, we have the map 
$\phi(z) = z^2$, so by Corollary~\ref{reducecor}, $\Phi_{N,\phi}^*(z)$ is irreducible if $2^N-1$ is prime. 
\end{proof}

\begin{remark}
For $N=2$, the polynomial $\Phi_2^*(x,y,a)$ has a factor of $a+1$, but since it is not defined for $a^2=1$ (this does not give a degree~$2$ map), this does not correspond to reducibility in the corresponding variety.  
\end{remark}

\begin{remark}
Of course, there are only $46$ known primes of the form $2^N-1$.  It seems likely that if $N$ is odd, then $M_2(N,\mathfrak C_2)$ is irreducible.  The curve cannot split as described for even $N$, and since the generic behavior of dynamic modular curves is irreduciblity, we expect that to be the case except when there is some clear reason for the curve to be reducible.  However, we have been unable to find a proof of this more general result.
\end{remark}

\subsubsection*{Acknowledgements}
The author thanks Joe Silverman for all of his help in every stage of this research, Steven J. Miller for his careful reading of an earlier draft,  Dan Abramovich for his careful reading, extensive comments, and suggestion of the term ``$h$-tuned dynatomic polynomials,'' and Jonathan Wise.  The author also thanks the reviewer for careful and thoughtful edits, which improved both the form and the content of the paper enormously.

\bibliographystyle{plain}	
\bibliography{refs}

\end{document}